\documentclass[reqno,12pt]{amsart}
\usepackage{amsmath,amssymb,latexsym,soul,cite,mathrsfs}
\usepackage{color,enumitem,graphicx}
\usepackage[colorlinks=true,urlcolor=blue,
citecolor=red,linkcolor=blue,linktocpage,pdfpagelabels,
bookmarksnumbered,bookmarksopen]{hyperref}
\usepackage[english]{babel}
\usepackage[left=2.6cm,right=2.6cm,top=2.9cm,bottom=2.9cm]{geometry}
\usepackage[hyperpageref]{backref}

\usepackage{import}

\numberwithin{equation}{section}

\newtheorem{theorem}{Theorem}[section]
\newtheorem{lemma}[theorem]{Lemma}

\newtheorem{proposition}[theorem]{Proposition}
\newtheorem{remark}[theorem]{Remark}
\newtheorem{definition}[theorem]{Definition}

\newenvironment{acknowledgement}{\noindent\textbf{Acknowledgments: }}{}
\newtheoremstyle{tttheorem}
{}                
{}                
{\slshape}        
{}                
{\bfseries}       
{'}               
{ }               
{}                
\theoremstyle{tttheorem}

\def\XXint#1#2#3{{\setbox0=\hbox{$#1{#2#3}{\int}$ }
		\vcenter{\hbox{$#2#3$ }}\kern-.6\wd0}}

\newcommand{\Ss}{\mathbf{S}}
\newcommand{\R}{\mathbf{R}}

\newcommand{\B}{\mathbf{B}}

\DeclareMathOperator{\Ric}{Ric}

\begin{document}
\title[An End to End Gluing Construction for $Q$-curvature]{An End to End Gluing Construction for Metrics of Constant $Q$-Curvature }

\author[A. S. Aiken]{A. Sophie Aiken}
\author[R. Caju]{Rayssa Caju}
\author[J. Ratzkin]{Jesse Ratzkin}
\author[A. Silva Santos]{Almir Silva Santos}

\address[A. S. Aiken]{Department of Mathematics, University of California, Santa Cruz
\newline\indent 
  1156 High St., Santa Cruz, CA 95064, USA}
\email{\href{mailto:ausaiken@ucsc.edu}{ausaiken@ucsc.edu}}

\address[R. Caju]{Department of Mathematical Engineering and Center for Mathematical Modeling (CNRS IRL 2807), University of Chile
\newline\indent 
    Beauchef 851, Edificio Norte, Santiago, Chile}
\email{\href{mailto:rayssacaju@gmail.com}{rcaju@dim.uchile.cl}}

\address[J. Ratzkin]{Department of Mathematics,
	Universit\"{a}t W\"{u}rzburg
	\newline\indent
	Emil-Fischer-Str. 40, 97070, W\"{u}rzburg-BA, Germany}
\email{\href{mailto:jesse.ratzkin@uni-wuerzburg.de}{jesse.ratzkin@uni-wuerzburg.de}}

\address[A. Silva Santos]{Department of Mathematics, 
	Federal University of Sergipe
	\newline\indent 
	49107-230, Sao Cristov\~ao-SE, Brazil}
\email{\href{mailto: almir@mat.ufs.br}{ almir@mat.ufs.br}}


\begin{abstract}
We produce many new complete, constant $Q$-curvature metrics on finitely punctured spheres 
by gluing together known examples. In our construction, we truncate one end of each summand and 
glue the two summands together ``end-to-end," where we've truncated them. We use this gluing 
construction to prove that the (unmarked) moduli space of solutions with $k$ punctures 
is topologically nontrivial provided $k \geq 4$. 
\end{abstract}

\maketitle

\tableofcontents

\section{Introduction} 

In the last 40 years many people have intensely studied singular 
curvature problems in Riemannian manifolds, particularly in the 
construction of constant curvature metrics with singularities on compact manifolds. 
In particular, the theory of constant scalar 
curvature metrics with isolated singularities is by now well-developed
and we give a (non-exhaustive) survey of the literature here. 
Caffarelli, Gidas and Spruck \cite{CGS} derived the asymptotic 
stucture of a singular point and later Korevaar, Mazzeo, Pacard and 
Schoen refined these asymptotics. Pollack \cite{Pol} proved that, under 
appropriate geometric conditions, the space of solutions is 
pre-compact and then Mazzeo, Pollack and Uhlenbeck investigated the 
local structure of the moduli space of solutions. Many people have 
constructed solutions using gluing methods; we particularly mention Schoen's construction 
\cite{Schoen1} and the construction of Mazzeo and Pacard \cite{MR1712628} here. 
In the present paper we explore this 
program for Branson's $Q$-curvature, which is a fourth-order generalization 
of scalar curvature. In particular, we construct many new and interesting 
examples of complete, constant $Q$-curvature metrics on finitely 
punctured spheres by gluing together known examples. As an 
application, we use this gluing construction to prove that the moduli 
space of solutions with exactly $k$ punctures is noncontractible, so long 
as $k \geq 4$. 

Let $(M,g)$ be a Riemannian manifold of dimension 
$n \geq 5$. The (fourth-order) $Q$-curvature of $g$ is defined as
\begin{equation*} \label{q_defn}
Q_g = -\frac{1}{2(n-1)} \Delta_g R_g - \frac{2}{(n-2)^2} |\Ric_g|^2 
+ \frac{n^3-4n^2 +16n-16}{8(n-1)^2(n-2)^2} R_g^2,
\end{equation*}
where $R_g$ and $\Ric_g$ are the scalar and Ricci curvatures and 
$\Delta_g$ is the Laplace-Beltrami operator. It is well-known that if 
$\widetilde g = u^{\frac{4}{n-4}} g$, for 
some positive function $u\in C^\infty(M)$, is a conformal metric to $g$, the $Q$-curvature transforms 
according to the rule 
    \begin{equation}\label{transformationlaw} 
        Q_{\widetilde g} = \frac{2}{n-4} u^{- \frac{n+4}{n-4}} P_gu, 
    \end{equation} 
    where the Paneitz operator $P_g$ is given by
    \begin{equation}\label{paneitz-branson} 
        P_gu=\Delta_g^2u + \operatorname{div}_g \left (\frac{4}{n-2} 
        \operatorname{Ric}_g (\nabla u, \cdot) -
        \frac{(n-2)^2 + 4}{2(n-1)(n-2)} R_g \nabla u \right )+\frac{n-4}{2} Q_g u.
    \end{equation} 
    The Paneitz operator transforms according to the law 
    \begin{equation} \label{transformationlaw2} 
    \widetilde{g} = u^{\frac{4}{n-4}} g \Rightarrow P_{\widetilde{g}} (\phi) = u^{-\frac{n+4}{n-4}}
    P_g(u\phi),
    \end{equation} 
    where $\phi$ is any smooth function. Observe that we recover \eqref{transformationlaw} 
    by plugging $\phi=1$ into \eqref{transformationlaw2}. 
    
After a short computation one sees that 
$$Q_{\overset{\circ}{g}} = \frac{n(n^2-4)}{8},$$ 
where $\overset{\circ}{g}$ is the usual round metric on the sphere $\Ss^n$. Using \eqref{transformationlaw} 
we see that $Q_{\widetilde g} = \frac{n(n^2-4)}{8}$ if and only if the conformal factor $u:M \rightarrow (0,\infty)$ 
satisfies the equation 
\begin{equation} \label{eq020} 
0 = \mathcal{N}_g (u) :=  P_g u - \frac{n(n-4)(n^2-4)}{16} u^{\frac{n+4}{n-4}}.
\end{equation} 
This is a fourth-order, elliptic PDE with critical Sobolev growth. 
Linearizing about the constant function $1$, we obtain the operator 
\begin{equation}\label{eq010}
    L_g(u)=P_gv-\frac{n(n+4)(n^2-4)}{16}v.
\end{equation}

We are particularly interested in the case that the background metric is 
the round metric on (a subdomain of) the sphere. In this case, the Paneitz 
operator factors as 
$$P_{\overset{\circ}{g}} = \left ( -\Delta_{\overset{\circ}{g}} + \frac{(n-4)(n+2)}{4} 
\right ) \circ \left ( -\Delta_{\overset{\circ}{g}} + \frac{n(n-2)}{4} \right ).$$
A theorem of C. S. Lin \cite{MR1611691} states that any global solution to the PDE 
$\mathcal{N}_{\overset{\circ}{g}} (U) = 0$ must arise from a M\"obius 
transformation. We can use stereographic projection to rephrase this 
theorem on Euclidean space. If $\Pi:\Ss^n \backslash \{ N\} \rightarrow \R^n$ is 
the standard stereographic projection then 
$$(\Pi^{-1})^* (\overset{\circ}{g}) = \frac{4}{(1+|x|^2)^2} g_{\text{euc}} =  
u_{\rm sph} ^{\frac{4}{n-4}} g_{\rm{euc}}, 
\qquad u_{\rm sph}(x) = \left ( \frac{1+|x|^2}{2} \right )^{\frac{4-n}{2}}.$$
We can now write the 
conformal metric $\widetilde g = U^{\frac{4}{n-4}} \overset{\circ}{g} = 
u^{\frac{4}{n-4}} g_{\rm euc}$, where $g_{\rm euc}$ is the standard Euclidean metric on $\R^n$ and 
$u = (U\circ \Pi^{-1}) u_{\rm sph}$. With 
respect to this Euclidean gauge, the conformally flat metric $\widetilde g = 
u^{\frac{4}{n-4}}g_{\rm euc}$ has constant $Q$-curvature if and only if 
\begin{equation} \label{flat_eq020} 
0 = \mathcal{N}_{g_{\rm euc}} (u) = \Delta^2 u -  \frac{n(n-4)(n^2-4)}{16} u^{\frac{n+4}{n-4}},
\end{equation} 
where $\Delta$ is the Laplace operator with respect to the standard flat metric $g_{\rm euc}.$

\begin{remark}\label{gauge_remark} We will often find it useful to transfer between the spherical gauge, 
writing $\widetilde{g} = U^{\frac{4}{n-4}} \overset{\circ}{g}$, and the Euclidean 
gauge, where $\widetilde{g} = u^{\frac{4}{n-4}} g_{\rm euc}$ without comment. Here and below 
we adopt the convention that conformal factors in the spherical gauge have captital letters 
and conformal factors in the Euclidean gauge have lower-case letters. The two are always 
related by $u = (U \circ \Pi^{-1})u_{\rm sph}$. Later on we will introduce a third useful gauge
near an isolated singular point, which we call the cylindrical gauge and is 
induced by the Emden-Fowler change of coordinates (see \eqref{emden_fowler_coords}). In this 
case the background metric is $g_{\rm cyl} = dt^2 + d\theta^2$, where $d\theta^2$ is the 
usual round metric on $\Ss^{n-1}$. Thereafter we will transfer between all three gauges, 
with the choice of background metric being clear from the context. 
\end{remark}

\noindent Lin \cite{MR1611691} proved that $u:\R^n \rightarrow (0,\infty)$ is a global 
solution of \eqref{flat_eq020} if and only if
$$ u(x) = u_{\lambda,x_0}(x) = \left (\frac{1+\lambda^2|x-x_0|^2}{2\lambda}\right)^{\frac{4-n}{2}}.$$ 
for some $\lambda >0$ and $x_0 \in \R^n$. Observe that 
when $\lambda \nearrow \infty$ the function $u_{\lambda,x_0}$ converges uniformly to zero outside 
any fixed neighborhood of $x_0$. This blow-up behavior motivates our examination of the 
following singular prescribed curvature problem: given a closed subset $\Lambda \subset \Ss^n$, 
find a conformal metric $g = U^{\frac{4}{n-4}} 
\overset{\circ}{g}$ with $\displaystyle Q_g = \frac{n(n^2-4)}{8}$ that is 
complete on $\Ss^n \backslash \Lambda$. The preceding paragraphs show that 
finding such a metric is equivalent to solving the boundary value 
problem 
\begin{equation} \label{sing_yam_bvp} 
U:\Ss^n \backslash \Lambda \rightarrow (0,\infty), \quad 
\mathcal{N}_{\overset{\circ}{g}} (U) = 0, \quad \liminf_{q \rightarrow \Lambda} 
U(q) = \infty.
\end{equation} 
Once again, we can use (inverse) stereographic projection to transfer this 
problem over to Euclidean space, transforming \eqref{sing_yam_bvp} 
into 
\begin{equation} \label{sing_yam_bvp2} 
u: \R^n \backslash \Gamma \rightarrow (0,\infty), \quad \mathcal{N}_{g_{\rm euc}}(u) 
= 0, \quad \liminf_{x \rightarrow \Gamma} u(x) = \infty, \quad \limsup_{|x| \rightarrow \infty} 
u(x) <\infty, 
\end{equation} 
where $\Gamma = \Pi(\Lambda)$ and (as before) $u = (U \circ \Pi^{-1})u_{\rm sph}$. We will 
refer to the metrics associated with solutions of \eqref{sing_yam_bvp} (equivalently, 
of \eqref{sing_yam_bvp2}) as singular Yamabe metrics

Our goal in the present manuscript is to construct new and interesting 
examples of singular Yamabe metrics where the singular set is finite by gluing 
together two known solutions. More specifically, we let $\Lambda_1 = \{ q_0^1, q_1^1, \dots, 
q_{k_1}^1\}$ and $\Lambda_2 = \{ q_0^2, q_1^2, \dots, q_{k_2}^2\}$ be two finite sets 
of $\Ss^n$ with $k_1+1$ and $k_2+1$ distinct points, respectively. We allow $q_j^1 = q_\ell^2$ 
for some pair of indices $j$ and $\ell$, but require $q_j^i \neq q_\ell^i$ for 
$i=1,2$ and $j \neq \ell$. Next we let $(M_1, g_1) = (\Ss^n \backslash \Lambda_1,  
U_1^{\frac{4}{n-4}} \overset{\circ}{g})$ and $(M_2, g_2) = (\Ss^n \backslash \Lambda_2, 
U_2^{\frac{4}{n-4}}\overset{\circ}{g})$ be two singular Yamabe metrics. In other 
words, 
\begin{equation} \label{summand_bvp} 
U_i : \Ss^n \backslash \Lambda_i \rightarrow (0,\infty), \quad 
\mathcal{N}_{\overset{\circ}{g}} (U_i) = 0, \quad \liminf_{x \rightarrow \Lambda_i} 
U_i(x) = \infty
\end{equation} 
for each $i=1,2$. Furthermore, we assume the asymptotic necksize (which we define below)
of $g_1$ at $q_0^1$ is 
the same as that of $g_2$ at $q_0^2$. We connect $(M_1,g_1)$ to $(M_2,g_2)$ by identifying 
two annuli about $q_0^1$ and $q_0^2$. We sketch an example with $k_1=3$ and $k_2=2$ in the 
cylindrical gauge. 

\begin{figure}[!ht]
\begin{center}
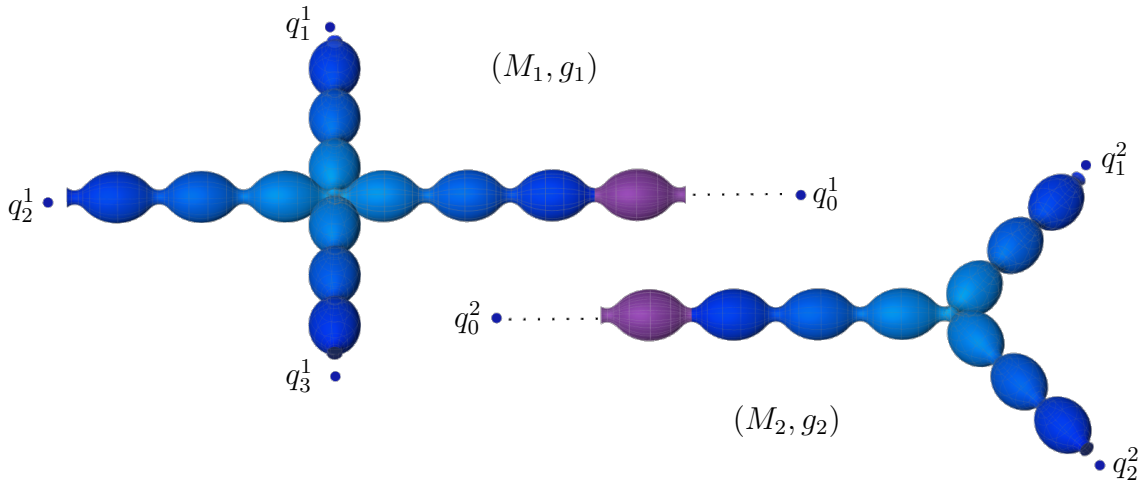
\caption{The two annuli we identify in $M_1$ and $M_2$.}\label{fig-approximate-solution}
\end{center}
\end{figure}

With this identification we 
obtain a new Riemannian manifold $(M,g)$ on $\Ss^n$ with $k_1+k_2$ punctures. This 
new manifold, which we sketch in Figure \ref{fig-approximate-solution-2}, has 
constant $Q$-curvature outside of a small annulus and the difference between 
$Q_g$ and $\displaystyle \frac{n(n^2-4)}{8}$ is small. Our goal is now 
to perturb $(M,g)$ to obtain a new metric with $k_1+k_2$ isolated singular points 
satisfying $\displaystyle Q_g = \frac{n(n^2-4)}{8}$. 
\begin{figure}[!ht]
\begin{center}
\begingroup%
  \makeatletter%
  \providecommand\color[2][]{%
    \errmessage{(Inkscape) Color is used for the text in Inkscape, but the package 'color.sty' is not loaded}%
    \renewcommand\color[2][]{}%
  }%
  \providecommand\transparent[1]{%
    \errmessage{(Inkscape) Transparency is used (non-zero) for the text in Inkscape, but the package 'transparent.sty' is not loaded}%
    \renewcommand\transparent[1]{}%
  }%
  \providecommand\rotatebox[2]{#2}%
  \newcommand*\fsize{\dimexpr\f@size pt\relax}%
  \newcommand*\lineheight[1]{\fontsize{\fsize}{#1\fsize}\selectfont}%
  \ifx\svgwidth\undefined%
    \setlength{\unitlength}{295.35160095bp}%
    \ifx\svgscale\undefined%
      \relax%
    \else%
      \setlength{\unitlength}{\unitlength * \real{\svgscale}}%
    \fi%
  \else%
    \setlength{\unitlength}{\svgwidth}%
  \fi%
  \global\let\svgwidth\undefined%
  \global\let\svgscale\undefined%
  \makeatother%
  \begin{picture}(1,0.38351822)%
    \lineheight{1}%
    \setlength\tabcolsep{0pt}%
    \put(0,0){\includegraphics[width=\unitlength,page=1]{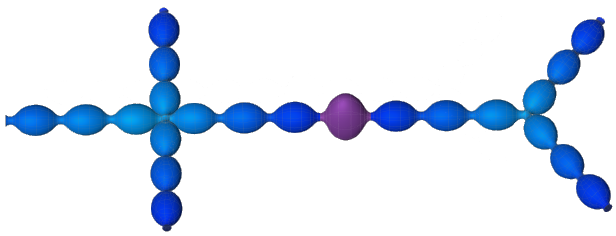}}%
    \put(0.56975542,0.09710011){\color[rgb]{0,0,0}\makebox(0,0)[t]{\lineheight{1.25}\smash{\begin{tabular}[t]{c}The gluing region\\$Q_{g_m}\not=\frac{n(n^2-4)}{8}$\end{tabular}}}}%
  \end{picture}%
\endgroup%

\caption{The approximate solution $(M,g)$ with 
 $k_1+k_2$ punctures.}\label{fig-approximate-solution-2}
\end{center}
\end{figure}

We require the summands $(M_i, g_i)$ to satisfy some hypotheses, which we 
describe in this paragraph, so 
that we can perturb the approximate solution to a constant $Q$-curvature 
metric. A result of Jin and 
Xiong \cite{JX} states that each of these metrics are asymptiotically radially 
symmetric near each of its punctures. Frank and K\"onig \cite{MR3869387} classified these 
radially symmetric solutions, showing that (after a change of variables 
we describe in Section \ref{delaunay_metrics}) they are all periodic, which allows us to 
assign an asymptotic necksize to each puncture. As we describe above, 
finding one of these singular Yamabe metrics is equivalent to solving a 
nonlinear PDE. As is usually the case, a key step in constructing a solution 
is to understand the mapping properties of the linearized operator. Let 
$g = u^{\frac{4}{n-4}} g_{\rm euc}$ be a complete, conformally flat metric on 
$\R^n \backslash \Gamma$ and write a nearby metric as 
$$g_v = \left ( 1+ \frac{v}{u} \right )^{\frac{4}{n-4}} g = (u+v)^{\frac{4}{n-4}}
g_{\rm euc}. $$
The condition that $g_v$ has $Q$-curvature equal to $\displaystyle{\frac{n(n^2-4)}{8}}$
is equivalent to the PDE
$$0 = \mathcal{N}_{g_{\rm euc}} (u+v) = \Delta^2 (u+v) - \frac{n(n-4)(n^2-4)}{16}
(u+v)^{\frac{n+4}{n-4}}. $$
The linearization of this operator about $u$, which one often calls the Jacobi 
operator, is 
\begin{equation}\label{eq011}
    {L}_u (v) = \Delta^2 v- \frac{n(n+4)(n^2-4)}{16} u^{\frac{8}{n-4}}v,
\end{equation}
and we call the metric $g=u^{\frac{4}{n-4}} g_{\rm euc}$ unmarked nondegenerate if 
${L}_u : C^{4,\alpha}_{\delta} (M,g) \rightarrow C^{0,\alpha}_{\delta} (M,g)$ 
is injective for all $\delta < -1$. One can equivalently phrase the nondegeneracy 
condition in terms of weighted Sobolev spaces, as in \cite{MPU} and \cite{cjs2024}.

\begin{theorem} \label{main_thm}
For $i=1,2$ consider $M_i=\Ss^n\backslash\{q_0^i,\ldots,q_{k_i}^i\}$, with $k_i\geq 2$. 
Let $(M_1, g_1)$ and $(M_2, g_2)$ be complete, constant $Q$-curvature metrics. We also 
assume the summands 
$(M_i, g_i)$ are both unmarked nondegenerate, that the asymptotic necksizes at 
$q_0^1$ and $q_0^2$ agree and that there exists a one-parameter family of constant 
$Q$-curvature metrics through $(M_1, g_1)$ that changes the necksize at $q_0^1$ to 
first order.  
There exists an $r_0>0$ such that for each $0 < r < r_0$ the gluing construction 
described above with annuli $\B_r(q _0^i)\backslash \B_{r/2}(q_0^i)$ for 
$i=1,2$ yields a complete constant $Q$-curvature metric with $k_1+ k_2$ punctures 
on $\Ss^n$. The metrics we produce in this gluing construction are all unmarked nondegenerate. 
\end{theorem} 

\begin{remark} 
    \begin{itemize} 
    \item The metrics constructed in \cite{del-ends2} are all unmarked nondegenerate 
    and belong to one-parameter families changing all necksizes, and so they all 
    satisfy the hypotheses of Theorem \ref{main_thm}.
    \item The term ``unmarked nondegenerate" refers to the nondegeneracy of the 
    constant $Q$-curvature metric in the unmarked moduli space $\mathcal{M}_k$
    of solutions with $k$ punctures. There is also a marked moduli space 
    $\mathcal{M}_\Lambda$ of solutions, each of whose singular set is $\Lambda$, 
    where $\Lambda$ is a specific $k$-tuple of points in $\Ss^n$. We equip both moduli spaces
    with the Gromov-Hausdorff topology. The difference 
    between families of metric in $\mathcal{M}_k$ and $\mathcal{M}_\Lambda$ is 
    that the position of the punctures remain fixed in a family of metrics in 
    $\mathcal{M}_\Lambda$, whereas the punctures are allowed to move in a family 
    of metrics in $\mathcal{M}_k$. 
    \item We observe that our gluing construction does 
    not {\it a priori} require small necksizes, as is often the case with 
    gluing constructions. On the other hand, the metrics we construct are 
    ``near the boundary of moduli space" in a different sense, in that the Riemannian 
    manifold we produce contains an annulus that is conformal to a long cylinder. In 
    this way, the conformal type of the metrics we produce is still near a 
    singular limit.
    \end{itemize}
\end{remark}

As is usually the case in a gluing construction, the key step is to 
construct a right-inverse for the linearized operator ${L}_u$. We can see the 
importance of this right inverse by writing out a formal Taylor series for 
the operator $\mathcal{N}$. Using the notation in the paragraph above, we 
see that we want to solve the PDE
\begin{equation} \label{goal_eqn1}
0 = \mathcal{N}_g \left ( 1+\frac{v}{u} \right ) = u^{-\frac{n+4}{n-4}}
\mathcal{N}_{g_{\rm euc}}(u+v) = u^{-\frac{n+4}{n-4}} (-\psi + {L}_u (v) + \mathcal{R}_u(v)),
\end{equation} 
where $\psi:=\mathcal{N}_{g_{\rm euc}}(u)$,  ${L}_u$ is the Jacobi operator discussed above,  
and $\mathcal{R}_u$ is the remainder term. We can rearrange \eqref{goal_eqn1} to 
read 
\begin{equation} \label{goal_eqn2} 
{L}_u(v) = -\mathcal{N}_{g_{\rm euc}} (u) - \mathcal{R}_u(v) = \psi - 
\mathcal{R}_u(v).
\end{equation} 
If ${G}_u$ is a right inverse of ${L}_u$, 
we can complete our construction by showing the map 
\begin{equation} \label{goal_eqn3} 
\mathcal{K}(v) := {G}_u(\psi-\mathcal{R}_u(v))
\end{equation} 
is a contraction under appropriate conditions. Indeed, if $v$ is a fixed 
point of $\mathcal{K}$ then applying ${L}_u$ on the left-hand side 
of \eqref{goal_eqn3} we obtain 
$${L}_u(v) = \psi - \mathcal{R}_u(v)$$
as desired. 

We complete the proof of Theorem \ref{main_thm} as follows. In 
Section \ref{sec:prelim} we recall some relevant, previous work and 
do some preliminary analysis. In Section \ref{approx_sltn} we 
construct an approximate solution and prove it has constant $Q$-curvature 
outside a compact set, and the $Q$-curvature is close to constant on 
this set. In Section \ref{sec:lin_anal} we perform linear analysis, 
particularly discussing the Jacobi operator of a Delaunay solution in 
Section \ref{sec-jacobi-operator} and constructing the right-inverse of the 
Jacobi operator of the approximate solution in Section \ref{sec:right_inverse}. 
In Section \ref{sec:nonlinear} we complete the gluing 
construction by proving the map $\mathcal{K}_u$ is a contraction and 
in Section \ref{sec:nondegen} we show that the solutions we construct are 
unmarked nondegenerate. 

\begin{remark}
One can modify our gluing construction by rotating the summand $(M_2, g_2)$ 
relative to $(M_1, g_1)$ by any fixed element $\phi \in SO(n-1)$. This action 
produces a smooth family of solutions, parameterized by the space of rotations $SO(n-1)$. 
\end{remark}

As an application of our gluing construction we prove the following. 
\begin{theorem} \label{noncontractible_thm}
If $k \geq 4$, the unmarked moduli space $\mathcal{M}_k$ 
is not contractible. 
\end{theorem} 
The same result holds in the scalar curvature setting. 

We prove this theorem by constructing a
submanifold in $\mathcal{M}_k$ that one cannot contract to a point by a 
homotopy. We start with a three-ended metric $(M_1, g_1)$ and a metric 
$(M_2, g_2)$ with $k-1$
punctures. We assume that both metrics are nondegenerate, that both metrics 
have one end with a common asymptotic necksize of $\varepsilon_*$, 
and that $(M_1, g_1)$ admits a Jacobi field deforming the necksize $\varepsilon_*$ 
to first order. We then glue the two metrics together end-to-end, along the 
end with the common asymptotic necksize of $\varepsilon_*$. Within this construction 
we have the freedom to rotate $(M_1, g_1)$ relative to $(M_2, g_2)$, using any element of 
$SO(n-1)$ fixing the common axis of the Delaunay asymptotes. We vary this rotation over the whole 
group, which then 
gives us a family of solutions in $\mathcal{M}_k$. In Section \ref{sec:top} we 
will complete the proof of Theorem \ref{noncontractible_thm} by 
showing the resulting submanifold cannot be contractible. The key tool we use is the 
forgetful map, sending a metric in $\mathcal{M}_k$ to its conformal 
class. \\
\\
\acknowledgement{Frank Pacard suggested this construction to us, and we thank him for 
the prompt. We also thank Jie Qing for helpful conversations about the space of conformal 
structures on a finitely punctured sphere and Pedro Gaspar for generously providing 
figures. ASA is supported by Centro de Modelamiento Matemático (CMM) BASAL fund FB210005 
for center of excellence from ANID-Chile, RC is supported by Fondecyt grant number 11230872 
and by Centro de Modelamiento Matemático (CMM) BASAL fund FB210005 for center of excellence 
from ANID-Chile and ASS is supported by 
CNPq grant number 408834/2023-4, 312027/2023-0, 444531/2024-6 and 403770/2024-6. Part of this 
research was completed while ASA visited the Center for Mathematical Modeling at the University 
of Chile, which we thank for their hospitality. }

\section{Preliminaries} \label{sec:prelim}

In this section we recall some useful material from other sources and perform 
some preliminary computations. 

\subsection{Delaunay metrics}
\label{delaunay_metrics}

The Delaunay metrics are all the constant $Q$-curvature metrics on a 
twice-punctured sphere and, as we will see later, play an important 
role in understanding the behavior of singular constant $Q$-curvature 
metrics with isolated singularities. 

Consider a metric $g = U^{\frac{4}{n-4}} \overset{\circ}{g}$ on $\Ss^n \backslash \{ p,q\}$ where $p$ 
and $q$ are distinct. After 
a rotation and a dilation, we can assume $p = N$ is the north pole and 
$q=S$ is the south pole. Transferring to the Euclidean gauge we write $g= u^{\frac{4}{n-4}} g_{\rm euc}$ 
where $u:\R^n \backslash \{ 0 \} \rightarrow (0,\infty)$. This metric has $Q$-curvature 
equal to $\frac{n(n^2-4)}{8}$ if and only if $u$ satisfies \eqref{flat_eq020} with 
a single singular point at the origin.

Frank and K\"onig \cite{MR3869387} classified 
all the solutions of
\begin{equation} \label{del_pde_eucl_coords}
u: \R^n \backslash \{ 0 \} \rightarrow (0,\infty), \qquad 
\Delta^2 u = \frac{n(n-4)(n^2-4)}{16} u^{\frac{n+4}{n-4}},
\end{equation}  
and we describe them here. First we perform the Emden-Fowler change of 
coordinates, defining 
\begin{equation} \label{emden_fowler_coords} 
\mathfrak{F} : \mathcal{C}^\infty (\B_r(0) \backslash \{ 0 \}) 
\rightarrow \mathcal{C}^\infty ((-\log r, \infty) \times \Ss^{n-1}), 
\qquad \mathfrak{F}(u) (t,\theta) = e^{\frac{4-n}{2}t} u(e^{-t}\theta). 
\end{equation} 
We can of course invert $\mathfrak{F}$, obtaining 
$$\mathfrak{F}^{-1}(v) (x) = |x|^{\frac{4-n}{2}} v(-\log |x|, \theta).$$
While the prefactor of $e^{\frac{4-n}{2}t}$ might 
look a little strange at first, a short computation shows it is geometrically 
necessary. Letting 
$$\Upsilon: \R \times \Ss^{n-1} \rightarrow \R^n \backslash 
\{ 0 \} , \qquad \Upsilon (t,\theta) = e^{-t} \theta,$$
we see 
$$\Upsilon^* (g_{\rm euc})  = e^{-2t} g_{\rm cyl}.$$ 
where $g_{\rm cyl}=dt^2+d\theta^2$ is the cylindrical metric. If we now consider a conformal metric $g= u^{\frac{4}{n-4}} g_{\rm euc}$, we 
see that 
$$\Upsilon^* (g)(t,\theta)  =  \mathfrak{F}(u) (t,\theta)^{\frac{4}{n-4}}
g_{\rm cyl} .$$

After the Emden-Fowler change of coordinates \eqref{del_pde_eucl_coords} becomes
\begin{equation} \label{del_pde_cyl_coords} 
\mathcal{N}_{\rm cyl} (v)=P_{\rm cyl} (v) - \frac{n(n^2-4)(n-4)}{16} v^{\frac{n+4}{n-4}}=0,
\end{equation}
where $v : \R \times \Ss^{n-1} \rightarrow (0,\infty)$. Using that $\Delta_{\rm cyl} = 
\partial_t^2 + \Delta_{\theta}$, where $\Delta_{\theta}$ is the Laplace-Beltrami 
operator on $\Ss^{n-1}$ with respect to the standard spherical metric, we have
\begin{align}
    \label{cyl_paneitz} 
P_{\rm cyl}  & = \Delta_{\rm cyl}^2 - \frac{n(n-4)}{2} \Delta_{\rm cyl} 
- 4 \partial_t^2 + \frac{n^2(n-4)^2}{16}\nonumber\\
& = \partial_t^4 + \Delta_\theta^2+ 2\Delta_\theta \partial_t^2
-\frac{n(n-4)+8}{2}  \partial_t^2 - \frac{n(n-4)}{2} \Delta_\theta + \frac{n^2(n-4)^2}{16}\nonumber
\end{align} is the Paneitz operator of the cylindrical 
metric. C. S. Lin \cite{MR1611691} used a
moving planes argument to prove that solutions 
of \eqref{del_pde_eucl_coords} are rotationally invariant, 
reducing \eqref{del_pde_cyl_coords} to the ODE 
\begin{equation} \label{del_ode} 
\ddddot v -  \frac{n(n-4)+8}{2} \ddot v + \frac{n^2(n-4)^2}{16} 
v - \frac{n(n^2-4)(n-4)}{16} v^{\frac{n+4}{n-4}}=0. 
\end{equation} 
Notice that 
\begin{equation}\label{eq002}
   \mathcal{H}_\varepsilon = - \dot v_\varepsilon\dddot v_\varepsilon + 
   \frac{1}{2} \ddot v_\varepsilon^2
+  \frac{n(n-4)+8}{4} \dot v_\varepsilon^2- \frac{n^2(n-4)^2}{32} 
v_\varepsilon^2 + \frac{(n-4)^2(n^2-4)}{32} v_\varepsilon^{\frac{2n}{n-4}}
\end{equation}
is a first integral for the ODE \eqref{del_ode}. 

\begin{theorem} [Frank and K\"onig \cite{MR3869387}] Let $\displaystyle 
\overline{\varepsilon} = \left ( \frac{n(n-4)}{n^2-4} \right) ^{\frac{n-4}{8}}$.
For each $\varepsilon \in (0,\overline{\varepsilon}]$ there exists a 
unique $v_\varepsilon\in C^4(\R)$ solving the 
ODE \eqref{del_ode} attaining its minimal value of $\varepsilon$ 
at $t=0$. All these solutions are periodic. 
Furthermore, let $v\in C^4(\R \times \Ss^{n-1})$ be a smooth 
solution of the PDE \eqref{del_pde_cyl_coords}. Then either 
$v(t,\theta) = (\cosh (t+T))^{\frac{4-n}{2}}$ for some $T \in \R$ 
or there exist $\varepsilon \in (0,\overline{\varepsilon}]$ and $T \in \R$
such that $v(t,\theta) = v_\varepsilon (t+T)$. 
\end{theorem}
One consequence of the classification by Frank and K\"onig is 
that the global, positive solutions of the ODE \eqref{del_ode} are 
ordered by the Hamiltonian energy $\mathcal{H}$ described in \eqref{eq002}. 

We can now write the Delaunay metric in Euclidean coordinates by 
reversing the coordinate transformation \eqref{emden_fowler_coords},
letting  
\begin{equation} \label{del_soln_eucl_coords} 
u_\varepsilon (x) = \mathfrak{F}^{-1}(v_\varepsilon)(x) = |x|^{\frac{4-n}{2}} v_\varepsilon (-\log |x|), 
\qquad g_\varepsilon = u_\varepsilon^{\frac{4}{n-4}} g_{\rm euc} 
= v_\varepsilon^{\frac{4}{n-4}} g_{\rm cyl} .
\end{equation} 
The geometric formulation of the Frank-K\"onig classification 
now reads: if $g = U^{\frac{4}{n-4}} \overset{\circ}{g}$ is a constant 
$Q$-curvature metric on $\Ss^n \backslash \{ p,q\}$ then, after a 
global conformal transformation, 
either $g$ extends smoothly to the round metric or $g$ is 
singular at both $p$ and $q$ and is the image of a Delaunay metric 
$g_\varepsilon$ after said conformal transformation. 

\subsection{Deformations of Delaunay metrics} 

We describe the geometric deformations of the Delaunay solutions to obtain a family of 
solutions to \eqref{del_pde_eucl_coords}. 

The first family comes from the scaling law associated to \eqref{flat_eq020}, which 
states that if $u$ solves \eqref{flat_eq020} then so does $x\mapsto R^{\frac{n-4}{2}} u(Rx)$ 
for each $R>0$. This family looks even simpler in Emden-Fowler coordinates. 
Using \eqref{del_soln_eucl_coords}, a short computation 
shows us 
\begin{equation*} \label{rescaled_del} 
R^{\frac{n-4}{2}} u_\varepsilon(Rx)  = |x|^{\frac{4-n}{2}} v_\varepsilon 
(-\log|x| -\log R). \end{equation*} 
Thus we see that the scaling law for \eqref{flat_eq020} is equivalent to the fact that if 
$v$ solves the ODE \eqref{del_ode} then so does any translate of $v$. 

The next family is also apparent from the construction of the Delaunay metrics, 
and is the family one obtains by varying the necksize parameter $\varepsilon$. We 
remark here that both these families preserve the rotational symmetry of the 
solution and the singular set. 

The next two families of deformations arise from translations. First we translate 
the origin, obtaining the solution 
$$\overline{u}_{\varepsilon, a} = u_\varepsilon (\cdot -a) ,$$
for any fixed $a \in \R^n$. 
The other family translates the point at infinity and is a bit more complicated to describe. 
We perform 
this transformation by first taking the Kelvin transform of $u_\varepsilon$, which 
interchanges $0$ and the point at infinity, then translating by $a$, and finally 
taking the Kelvin transform again. Recall Kelvin's transformation law for the 
bilaplacian, namely 
\begin{equation} \label{kelvin_law}
\mathbb{K}(u) (x) = |x|^{4-n} u\left ( x|x|^{-2}\right ).
\end{equation} 
It is a straight-forward computation to see that $\Delta^2 (\mathbb{K}(u))(x) = |x|^{-8} 
\mathbb{K} (\Delta^2 u)(x)$. Using \eqref{kelvin_law} we see that if $u$ solves \eqref{flat_eq020} then 
so does $\mathbb{K}(u)$. Now we define 
\begin{eqnarray} \label{translated_del}
 \mathbb{K} (\mathbb{K} (u_\varepsilon) (\cdot - a)) 
& = & |x|^{\frac{4-n}{2}} \left | \frac{x}{|x|} - |x| a\right |^{\frac{4-n}{2}} 
v_\varepsilon \left( -\log |x|+\log \left | \frac{x}{|x|} - |x| a\right | \right ).
\end{eqnarray} 
We can put the first and last of these deformations together, defining 
\begin{equation} \label{deformed_del} 
u_{\varepsilon,R,a} (x) = |x|^{\frac{4-n}{2}} \left | \frac{x}{|x|} - |x|a
\right |^{\frac{4-n}{2}} v_\varepsilon\left ( -\log |x| + \log \left | \frac{x}{|x|}
- |x|a\right |- \log R\right ). 
\end{equation} 
We notice that the function \eqref{deformed_del} has two singular points, namely $x=0$ and 
$x=a/|a|^2$, but that the point at infinity is now a regular point for the metric. In 
Emden-Fowler coordinates \eqref{deformed_del} takes the form 
\begin{equation*} \label{translated_del2} 
v_{\varepsilon,T,a} (t,\theta) := |\theta - e^{-t} a|^{\frac{4-n}{2}} v_\varepsilon 
(t+T+ \log |\theta - e^{-t} a|),
\end{equation*} 
with $T=-\log R$. For simplicity, we denote $u_{\varepsilon,R}:=u_{\varepsilon,R,0}$, 
$u_{\varepsilon,a}:=u_{\varepsilon,1,a}$, $v_{\varepsilon,a}:=v_{\varepsilon,0,a}$ and 
$v_{\varepsilon,T}:=v_{\varepsilon,T,0}$. Note that $v_{\varepsilon,T,a}(t)=
v_{\varepsilon,a}(t+T)$ and $v_{\varepsilon, T+T_\varepsilon,a}(t)=v_{\varepsilon, T, a}(t)$ 
where $T_\varepsilon$ is the period of the $v_\varepsilon$. By \cite[Proposition 2.4]{MR4778469} 
it holds
\begin{align}
    u_{\varepsilon,R,a}(x) & =u_{\varepsilon,R}(x)+((n-4)u_{\varepsilon,R}(x)
    +|x|\partial_ru_{\varepsilon,R}(x))\langle a,x\rangle+\mathcal O(|a|^2|x|^{\frac{8-n}{2}}) 
    \nonumber\\
    & =u_{\varepsilon,R}(x)+\left(\frac{n-4}{2}v_\varepsilon(-\log|x|-\log R)
    -\dot v_\varepsilon(-\log|x|-\log R)\right)|x|^{\frac{4-n}{2}}\langle a,x\rangle\nonumber\\
    & +\mathcal O(|a|^2|x|^{\frac{8-n}{2}})\label{trans_del_expansion}
\end{align}

Once again, this looks simpler in Emden-Fowler coordinates. We have 
\begin{equation} \label{trans_del_expansion2} 
v_{\varepsilon,T,a} (t,\theta) = v_{\varepsilon,T} (t) + e^{-t} \langle \theta, a\rangle 
\left ( \frac{n-4}{2} v_{\varepsilon,T} (t) - \dot v_{\varepsilon,T} (t) \right ) + \mathcal{O} 
(|a|^2e^{-2t}).
\end{equation} 
We observe that the expansions \eqref{trans_del_expansion} and \eqref{trans_del_expansion2} are
informative when $|x| \searrow 0$, respectively $t \rightarrow \infty$. 

At this point we observe that, for each $a \in \R^n$ the 
function $$\chi_a : {\Ss}^{n-1} \rightarrow \R, \qquad \chi_a(\theta) 
= \langle \theta, a \rangle $$ 
is an eigenfunction of $\Delta_\theta$ with eigenvalue $n$. We will see later 
that this is a useful fact.

\subsection{Asymptotics}\label{sec-asymptotics}

The following asymptotic statement holds. 

Let $u:\R^n \backslash \Gamma \rightarrow (0,\infty)$ satisfy \eqref{flat_eq020} with 
singularities at $\Gamma=\{p_0,\ldots,p_k\}$. Let $p_j \in \Gamma$. It was proved 
in \cite{jesse2020} that there exist $\varepsilon_j \in (0, \overline{\varepsilon}]$, 
$T_j\in \R$ and $a_j \in \R^n$ such that 
\begin{equation} \label{asymp_emden_fowler} 
\mathfrak{F} (u (\cdot - p_j) ) (t, \theta) = 
v_{\varepsilon_j,T_j, a_j} (t, \theta) + \mathcal{O}(e^{-\beta_j t})
\end{equation} 
for some $\beta_j > 1$ as $t \rightarrow \infty$. We can write this asymptotic 
expansion in the Euclidean gauge as 
\begin{eqnarray*} 
u(x-p_j) & = & |x|^{\frac{4-n}{2}} \left ( \left | \frac{x}{|x|} - |x| a_j \right |
^{\frac{4-n}{2}} v_{\varepsilon_j} \left ( -\log |x| + \log \left | 
\frac{x}{|x|} - |x|a_j \right | + T_j \right )+ \mathcal O(|x|^{\beta_j})\right ) \\ 
& = & u_{\varepsilon_j, R_j, a_j} (x) + \mathcal O(|x|^{\frac{4-n}{2} + \beta_j}),
\end{eqnarray*} 
where $R_j = e^{-T_j}$. 

This now allows us to assign asymptotic data to each puncture $p_j \in\Gamma$, namely an 
asymptotic necksize $\varepsilon_j>0$, a scaling parameter $T_j$ and a translation 
parameter $a_j\in\R^n$. Since $v_\varepsilon$ is periodic, 
$v_\varepsilon(t+T)=v_\varepsilon(t+T+mT_\varepsilon)$, for any $m\in\mathbb Z$. Then we 
can consider $T<0$ so that $v_\varepsilon$ achieves its minimum at $t=-T>0$. Recall 
that $T_\varepsilon>0$ is the period of the Delaunay solution $v_\varepsilon$.

\section{The Approximate Solution}
\label{approx_sltn}

We begin with two Riemannian manifolds $(M_1,g_1)$ and $(M_2,g_2)$, where 
$M_i=\Ss^n\backslash\Lambda_i$ and $g_i=U_i^{\frac{4}{n-4}}\mathring {g}$, 
for $i=1,2$. Here, $\Lambda_i=\{q_0^i,\ldots,q^i_{k_i}\}$, for $i=1,2$ denotes a finite 
set of points on the sphere $\Ss^n$, and $\mathring g$ represents the standard metric 
on $\Ss^n$. We assume that both $g_1$ and $g_2$ are complete metrics with $Q$-curvature 
equal to $(n^2-4)/8$. Each conformal factor $U_i$ satisfies \eqref{summand_bvp}. Define 
the radius $r_0$ by
$$r_0<\frac{1}{4}\min\{\inf\limits_{j\neq l}dist_{\mathring g}(q_j^i,q_l^i):i\in\{1,2\}\}.$$ 
Equivalently we write the conformal factors in the Euclidean gauge as $u_i = (U_i \circ 
\Pi^{-1}) \cdot u_{\rm sph}$. 
Then for each $i\in\{1,2\}$, the balls $\B_{r_0}(q_j^i)$ are disjoint. In normal 
coordinates centered at $q_j^i$, using \eqref{asymp_emden_fowler} we observe that the 
conformal factor satisfies the asymptotics 
\begin{equation}\label{eq018}
    \mathfrak{F} (u_i  ) (t, \theta) = v_{\varepsilon_j^i, a_j^i} (t-T_j^i, \theta) 
    + \mathcal{O}(e^{-\beta_j^i t}),
\end{equation}
for some parameters $\varepsilon_j^i>0$, $T_j^i>0$, $a\in\R^n$ and $\beta_j^i>1$. Given that 
the necksizes at $q_0^1$ and $q_0^2$ are assumed to be identical, we define 
$\varepsilon_*:=\varepsilon_0^1=\varepsilon_0^2$.

Let $m \in \mathbb{N}$ and define the annuli 
$$A_{1,m} = \mathbf{B}_{r_0e^{-T_0^1 -m T_{\varepsilon_*}}}(q_0^1)\backslash 
\overline{\mathbf{B}_{r_0e^{-T_0^1 -(m+1) T_{\varepsilon_*}}}(q_0^1)}, \quad 
A_{2,m} = \mathbf{B}_{r_0e^{-T_0^2 -m T_{\varepsilon_*}}}(q_0^2)\backslash 
\overline{\mathbf{B}_{r_0e^{-T_0^2 -(m+1) T_{\varepsilon_*}}}(q_0^2)}.$$
We introduce cylindrical coordinates in the punctured balls $\mathbf{B}_{r_0}(q_0^1)\backslash
\{q_0^1\}$ and $\mathbf{B}_{r_0}(q_0^2)\backslash\{q_0^2\}$ as follows
\begin{itemize}
    \item In $\mathbf{B}_{r_0}(q_0^1)\backslash\{q_0^1\}$, let $t=-\log\left(
    \frac{dist_{\mathring g}(x,q_0^1)}{r_0}\right)$.
    \item In $\mathbf{B}_{r_0}(q_0^2)\backslash\{q_0^2\}$, let $\tau=-\log\left(
    \frac{dist_{\mathring g}(x,q_0^2)}{r_0}\right)$.
\end{itemize}
Consequently, each punctured ball can be identified with the half cylinder $(0,\infty)\times 
\Ss^{n-1}$. Note that $x \in A_{1,m}$ if and only if $-\log\left(dist_{\mathring g}
(x,q_0^1)r_0^{-1}\right) = t$
for $T_0^1+m T_{\varepsilon_*} \leq t \leq T_0^1+(m+1)T_{\varepsilon_*}$, and similarly
$y \in A_{2,m}$ if and only if $-\log\left(dist_{\mathring g}(y,q_0^2)r_0^{-1}\right) = \tau$ 
for $T_0^2+mT_{\varepsilon_*} \leq \tau \leq T_0^2+(m+1)T_{\varepsilon_*}$. We now proceed to 
identify the two annuli $A_{1,m}$ and $A_{2,m}$ according to the rule that 
\begin{equation} \label{approx_soln_identification} 
x \sim y \Leftrightarrow t+\tau = T_0^1 + T_0^2 + (2m+1)T_{\varepsilon_*},
\end{equation} 
where $t$ and $\tau$ are as described above, and we denote the common annulus as $\mathbb{A}_m$. 
Observe that after this quotient we identify the inner boundary of $A_{1,m}$ with the outer 
boundary of $A_{2,m}$ 
and vice-versa, see the Figure \ref{fig-cutoff-functions}. In the spherical gauge we can 
write \eqref{approx_soln_identification}
as 
\begin{equation*} \label{approx_soln_identification2} 
dist_{\mathring g}(x,q_0^1)\cdot dist_{\mathring g}(y,q_0^2) =
(r_0)^2 e^{-T_0^1-T_0^2 - (2m+1)T_{\varepsilon_*}}. 
\end{equation*} 

We define the extended annulus 
\begin{equation}\label{eq023}
    \widehat{\mathbb{A}}_m = \mathbf{B}_{r_0} (q_0^1) \backslash \overline{
\mathbf{B}}_{r_0e^{-T_0^1 -m T_{\varepsilon_*}}}(q_0^1) \cup \mathbb{A}_m 
\cup \mathbf{B}_{r_0} (q_0^2) \backslash \overline{
\mathbf{B}}_{r_0e^{-T_0^2 -m T_{\varepsilon_*}}}(q_0^2),
\end{equation}
where $r_0$ is defined as above.
This extended annulus will prove useful in later computations. In Emden-Fowler coordinates, 
we can identify $\widehat{\mathbb{A}}_m$ with $(-T_0^1-(m+1/2)
T_{\varepsilon_*},T_0^2+(m+1/2)T_{\varepsilon_*})\times\Ss^{n-1}$. 

In the common annulus $\mathbb{A}_m$, we can write $g_1$ in Emden-Fowler 
coordinates as 
\begin{equation} \label{annulus_metric1} 
g_1 = v_1(t,\theta)^{\frac{4}{n-4}} (dt^2 + d\theta^2), \qquad v_1(t,\theta)
=v_{\varepsilon_*, a_0^1} (t-T_0^1, \theta) + \mathcal{O}(e^{-\beta_0^1 t}),\end{equation} 
while $g_2$ in Emden-Fowler coordinates has the form 
\begin{align} \label{annulus_metric2}
g_2 & =  v_2(\tau,\theta)^{\frac{4}{n-4}} (d\tau^2+d\theta^2) \\ \nonumber 
& =   (v_2(-t+T_0^1+T_0^2+(2m+1)T_{\varepsilon_*}, \theta))^{\frac{4}{n-4}} 
(dt^2+d\theta^2), \\ \nonumber 
v_2(\tau,\theta) & =  v_{\varepsilon_*, a_0^2} (\tau-T_0^2, \theta) + 
\mathcal{O}(e^{-\beta_0^2 \tau}) \nonumber
\end{align}
where $t$ and $\tau$ are related by \eqref{approx_soln_identification}.

By applying a global conformal transformation of the round sphere, we can adjust the metric 
$(M_i,g_i)$ at the point $q_0^i$. It can be readily verified that the metrics \eqref{annulus_metric1} 
and \eqref{annulus_metric2} satisfy
$$\mathfrak{F}^{-1}(v_i)(x)=u_{\varepsilon_*,R,a_0^i}(x)+\mathcal O(|x|^{\frac{4-n}{2}+\beta_0^i}).$$
Finally, using the transformation \eqref{translated_del}, we translate by $-a_0^i$ to obtain that 
the conformal factor has the expansion $u_{\varepsilon_*,R}(x)+
\mathcal O(|x|^{\frac{4-n}{2}+\beta_0^i})$ around $q_0^i$. Note that this construction results 
in a slight shift in the positions of the points $q_1^i,\ldots,q_{k_i}^i$. For convenience, we 
will maintain the original notation $q_j^i$ for these points. Therefore, we can assume that
\begin{equation}\label{eq001}
    v_i(t,\theta)=v_{\varepsilon_*}(t-T_0^i)+w_0^i(t,\theta),
\end{equation}
with $w_0^i(t,\theta)=\mathcal{O}(e^{-\beta_0^i t})$. This implies that $\|w_0^i\|_{C^{4,\alpha}
(\mathbb A_m)}=\mathcal O(e^{-\beta_0^imT_{\varepsilon_*}})$.

Finally we let $\chi$ be a smooth function satisfying 
\begin{equation}\label{eq003}
    \chi: \R \rightarrow [0,1], \qquad \chi(t) = \left \{ \begin{array}{rl} 
1 & t \leq T_0^1+(m+1/4)T_{\varepsilon*} \\ 0 & t \geq T_0^1+(m+3/4)T_{\varepsilon*}\end{array} \right. 
\end{equation}
and define the approximate solution on the annulus $\mathbb{A}_m$ by 
\begin{align}
v_m(t,\theta) & =  \chi( t) v_1(t,\theta) + (1-\chi(t))
v_2(\tau,\theta) \nonumber \\ \nonumber 
&= \chi (t) v_1(t,\theta) + (1-\chi(t))
v_2(-t+T_0^1+T_0^2+(2m+1)T_{\varepsilon_*},\theta)\\
&= v_{\varepsilon_*}(t-T_0^1)+\chi (t) w_0^1(t,\theta) + (1-\chi(t))
w_0^2(-t+T_0^1+T_0^2+(2m+1)T_{\varepsilon_*},\theta)\nonumber\\
& = v_{\varepsilon_*}(t-T_0^1) +\mathcal O(e^{-\beta mT_{\varepsilon_*}}),\label{annulus_metric3}
\end{align} 
where $\beta=\min\{\beta_0^1,\beta_0^2\}>0$. Using this new function $v_m$ as a conformal factor 
in the annulus $\mathbb{A}_m$ we obtain the metric $g_m = v_m^{\frac{4}{n-4}} (dt^2 + d\theta^2)$. 
This procedure yields a smooth metric on the connected sum
$$M = (M_1 \backslash \B_{r_0e^{-T_0^1-(m+1)T_{\varepsilon_*}}}(q_0^1)) \cup (M_2 \backslash 
\B_{r_0e^{-T_0^2-(m+1)T_{\varepsilon_*}}}(q_0^2)) /\sim,$$ 
where $\sim$  denotes the identification described in \eqref{approx_soln_identification}. 
Consequently, $g_m$ can be identified as a complete metric on $\Ss^n \backslash \{ q_1^1, \dots, 
q^1_{k_1}, q^2_1, \dots, q^2_{k_2}\}$.  

We now make some remarks regarding our approximate solutions. First, observe
that \eqref{annulus_metric1} and \eqref{annulus_metric2} parameterize the annulus 
$\mathbb{A}_m$ with reversed orientations, which is compatible with the identification 
\eqref{approx_soln_identification}. Second, since the approximate solution $g_m$ coincides 
with $g_i$ in $M_i\backslash \B_{r_0e^{-T_0^i-mT_{\varepsilon_*}}}(q_0^i)$, we conclude that 
$$Q_{g_m}=\frac{n(n^2-4)}{8}+\psi,$$
for some smooth function $\psi$ with compact support in $\mathbb A_m$.

\begin{lemma}\label{lem001}
 $\|\psi\|_{C^{4,\alpha}(M)}=\mathcal O(e^{-m\beta T_{\varepsilon_*}})$, for some $\beta>1$.
\end{lemma}
\begin{proof} We have 
\[g_m(t,\theta)=\bigg(v_{\varepsilon_*}(t-T_{0}^1)(1+v(t,\theta))\bigg)^{\frac{4}{n-4}}
g_{\rm cyl}=\left(1+v(t,\theta)\right)^{\frac{4}{n-4}}g_{\varepsilon_*}.\]
where $v(t, \theta)=\mathcal O(e^{-m\beta T_{\varepsilon_*}})$,
since $v_{\varepsilon_*}$ is periodic. Recall that the $Q$-curvature of the Delaunay metric 
\(g_{\varepsilon_*}=v_{\varepsilon_*}^{\frac{4}{n-4}}g_{\rm cyl}\) is equal to \(n(n^2-4)/8\). 
Thus, $\mathcal N_{g_{\varepsilon_*}}(1)=0$, where $\mathcal N_{g_{\varepsilon_*}}$ is given 
by \eqref{eq020}. Thus
\begin{align*}
    \frac{n-4}{2}\psi(1+v)^{\frac{n+4}{n-4}}
    &=\mathcal N_{g_\varepsilon}(1+v)= L_{g_\varepsilon}(v)+R(v)
\end{align*}
where $L_{g_{\varepsilon_*}}$ is given in \eqref{eq010}
and $R(v)=\mathcal O(v^2)$. This implies the result.
\end{proof}

\section{Linear Analysis} \label{sec:lin_anal}

In this section, our goal is to discuss the core of the paper: the invertibility of the 
linearized operator at the approximate solution constructed in the previous section, within 
appropriate function spaces.

\subsection{Function spaces} 

We can identify the punctured balls $\mathbf{B}_{r_0}(q_j^i)\setminus\{q_j^i\}$ with the 
half cylinder $(0,+\infty)\times\Ss^{n-1}$ similar to the identification at the beginning of 
Section \ref{approx_sltn}. 

For the next definition consider $\displaystyle M_i^c=\Ss^n\backslash
\bigcup_{j=0}^{k_i}B_{r_0}(q_j^i)$. Note that we can write
\begin{equation}\label{eq022}
    M={M_1^c}\cup \left(\bigcup_{j=1}^{k_1}B_{r_0}(q_j^1)\setminus\{q^1_j\}\right)\cup {M_2^c}
    \cup \left(\bigcup_{j=1}^{k_2}B_{r_0}(q_j^2)\setminus\{q^2_j\}\right)\cup \widehat{\mathbb{A}}_m.
\end{equation}
Also, we will consider the identification 
$$\widehat{\mathbb A}_m\sim(-T_0^1-(m+1/2)T_{\varepsilon_*},
T_0^2+(m+1/2)T_{\varepsilon_*})\times\Ss^{n-1}.$$

\begin{definition}
    Given $\delta\in\R$, define $C_\delta^{k,\alpha}(M_i)$ and $C_\delta^{k,\alpha}
    (\widehat{\mathbb A}_m)$ as the sets of functions $u$ such that the norms
    $$\|u\|_{C^{k,\alpha}_\delta(M_i)}:=\|u\|_{C^{k,\alpha}(M_i^c)}+\max_{0\leq j\leq k_i}
    \sup_{t_j\geq 1}e^{-\delta t_j}\|u\|_{C^{k,\alpha}((t_j-1,t_j+1)\times\Ss^{n-1})}$$
    and
    $$\|u\|_{C_\delta^{k,\alpha}(\widehat{\mathbb A}_m)}=
    \sup\left(\frac{\cosh^\delta(mT_\varepsilon)}{\cosh^\delta s}
    \left\|u\right\|_{C^{k,\alpha}([s-1,s+1]\times\Ss^{n-1})}\right)$$
    are finite, respectively. The supremum is taken over $s$ such that $[s-1,s+1]
    \subset (-T_0^1-(m+1/2)T_{\varepsilon_*},T_0^2+(m+1/2)T_{\varepsilon_*})$. Then we 
    define $C_\delta^{k,\alpha}(M)$ to be the space of functions such that the norm
    \begin{align*}
        \|u\|_{C^{k,\alpha}_\delta(M)}  
 & =\|u\|_{C^{k,\alpha}({M_1^c})}+ \|u\|_{C^{k,\alpha}({M_2^c})}+\max_{1\leq j\leq k_1}
 \sup_{t_j\geq 1}e^{-\delta t_j}\|u\|_{C^{k,\alpha}((t_j-1,t_j+1)\times\Ss^{n-1})}\\
   &   +\max_{1\leq j\leq k_2}\sup_{\tau_j\geq 1}e^{-\delta \tau_j}
   \|u\|_{C^{k,\alpha}((\tau_j-1,\tau_j+1)\times\Ss^{n-1})}+ 
   \|u\|_{C_\delta^{k,\alpha}(\widehat{\mathbb A}_m)}
    \end{align*}
    is finite.
\end{definition}

\subsection{The Jacobi operator on a Delaunay solution}\label{sec-jacobi-operator}

The linearization of the operator $\mathcal N_{g_{\rm euc}}$ given in \eqref{flat_eq020} about 
a Delaunay solution $u_\varepsilon$ will be denoted as $L_\varepsilon$ and is given by
$$L_\varepsilon = \Delta^2 - \frac{n(n+4)(n^2-4)}{16} 
u_\varepsilon^{\frac{8}{n-4}}$$
or differentiating \eqref{del_pde_cyl_coords} about the corresponding Delaunay 
solution $v_\varepsilon$ we get
\begin{align} \label{del_linearization2} 
\mathcal{L}_\varepsilon & =  \partial_t^4 + \Delta_{\theta}^2 + 
2\Delta_{\theta} \partial_t^2 - \frac{n(n-4)}{2} \Delta_{\theta} \\ \nonumber 
&- \frac{n(n-4)+8}{2} \partial_t^2 + \frac{n^2(n-4)^2}{16} - \frac{n(n+4)(n^2-4)}{16} 
v_\varepsilon^{\frac{8}{n-4}}.
\end{align} 
These operators are related by
$$\mathcal{L}_\varepsilon (w) (t,\theta) = e^{\frac{4-n}{2} t}
L_\varepsilon (\mathfrak{F}^{-1} (w)) \circ \Upsilon (t,\theta).$$

A general solution of $L_\varepsilon(v)=0,$ is called a Jacobi field. It is easy to see that 
if $s\mapsto v_s$ is a differentiable path of solutions of $\mathcal{N}_{g_{\rm euc}}(v_s)=0$, 
then $(\partial/\partial s)|_{s=0}v_s$ is a Jacobi field. In this way, taking the derivative of 
the family of solutions $u_{\varepsilon,R,a}$ with respect to one parameter, we obtain a Jacobi 
field. First, we consider the one parameter families $\varepsilon\mapsto u_{\varepsilon}$ and 
$R\mapsto u_{\varepsilon,R}$. The derivatives of these families of solutions give us solutions to
$$L_\varepsilon(w_\varepsilon^{0,\pm})=0,$$
where
\begin{align*}
    w_\varepsilon^{0,-}(x) & =\left.\frac{d}{d\eta}\right|_{\eta=0}u_{\varepsilon+\eta}(x)
    =|x|^{\frac{4-n}{2}}\left.\frac{d}{d\eta}\right|_{\eta=0}v_{\varepsilon+\eta}(-\log|x|)
    =|x|^{\frac{4-n}{2}}v_{\varepsilon}^{0,-}(-\log|x|),\\
    w_\varepsilon^{0,+}(x) & = \left.\frac{d}{dR}\right|_{R=1}u_{\varepsilon,R}(x)
    =|x|^{\frac{4-n}{2}}\dot v_\varepsilon(-\log |x|)=|x|^{\frac{4-n}{2}}v_\varepsilon^{0,+}(-\log|x|).
\end{align*}

Differentiating the relation $v_\varepsilon(t+T_\varepsilon)=v_\varepsilon(t)$, it is 
straightforward to verify that $v_\varepsilon^{0,+}$ is bounded and periodic, while 
$v_\varepsilon^{0,-}$ grows linearly, see \cite[Lemma 8]{jesse2020}. Here, $T_\varepsilon$ is 
the period of $v_\varepsilon$. Consider $(\phi_j,\lambda_j)$ the eigendata of the Laplacian
on $\Ss^{n-1}$, that is, $\Delta_{\theta}\phi_j+\lambda_j\phi_j=0$. The Jacobi fields 
$v_\varepsilon^{0,\pm}$ correspond to $\lambda_0=0$, considering $\phi_0\equiv 1$. The next 
Jacobi fields correspond to the eigenvalues $\lambda_1=\cdots=\lambda_n=n-1$, where the 
corresponding eigenfunctions are $\phi_j(\theta)=\theta_j$. Using \eqref{trans_del_expansion}, these are
\begin{align*}
    w_\varepsilon^{j,+}(x) & = \left.\frac{\partial}{\partial a_j}\right|_{a=0}
    u_{\varepsilon,1,a}(x)=|x|^{\frac{6-n}{2}}\left(\frac{n-4}{2}v_\varepsilon(-\log|x|)
    -\dot v_\varepsilon(-\log|x|)\right)\phi_j(\theta)\\
    & =|x|^{\frac{4-n}{2}}v_\varepsilon^{j,+}(-\log|x|)\phi_j(\theta),\\
    w_\varepsilon^{j,-}(x) & = \left. \frac{\partial}{\partial a_j} \right |_{a=0}u_\varepsilon
    (\cdot - a) =|x|^{\frac{2-n}{2}}\left(\frac{4-n}{2}v_\varepsilon(-\log|x|)
    -\dot v_\varepsilon(-\log|x|)\right)\phi_j(\theta)\\
    & =|x|^{\frac{4-n}{2}}v_\varepsilon^{j,-}(-\log |x|)\phi_j(\theta).
\end{align*}

Notice in particular that the Jacobi fields $v_\varepsilon^{j,\pm}=\mathcal O(e^{\mp t})$ as $t\to\infty$. We can use separate variables and write 
$$v(t,\theta)=\sum v_j(t)\phi_j(\theta),$$
In this case, $v_j$ satisfies the ordinary differential equation
	\begin{align*}
		\mathcal L_{\varepsilon,j}(v_j):=\ddddot {v_j}-&\left(2\lambda_j+\frac{n(n-4)+8}{2}
		\right)\ddot v_j	\\
  &+\left(\lambda_j^2+\frac{n(n-4)}{2}\lambda_j
		+\frac{n^2(n-4)^2}{16}-\frac{n(n+4)(n^2-4)}{16}
		v_\varepsilon^{\frac{8}{n-4}}\right)v_j=0.
		\end{align*}

The {\bf indicial roots} of \eqref{del_linearization2} are the exponential growth rates of 
the solutions to $\mathcal L_{\varepsilon,j}(v)=0$. It was proved in \cite{jesse2020} that 
the set of indicial roots is given by
$$\Gamma_\varepsilon=\{\ldots,-\gamma_{\varepsilon,2},-\gamma_{\varepsilon,1}
=-1,0,1=\gamma_{\varepsilon,1},\gamma_{\varepsilon,2},\ldots\},$$
where $\gamma_{\varepsilon,j}<\gamma_{\varepsilon,j+1}\to+\infty$.

\subsection {Linear analysis on the approximate solution} \label{sec:right_inverse}

\begin{definition}\label{definition-defi-space}
The deficiency space $\mathcal W_{g_i}$ of $(M_i,g_i)$ is defined as
$$\mathcal W_{g_i}=\mbox{Span}\{\chi_iv_{\varepsilon_j^i}^{l,\pm}\phi_l:l=0,\ldots,n\mbox{ and }
j=0,\ldots,k_i\},$$
where $\chi_i$ is a fixed cutoff function equaling one in a small ball around each singular 
point and vanishing outside a slightly larger ball.
The deficiency space $\mathcal W_{g_m}$ of $(M,g_m)$ is defined as
$$\mathcal W_{g_m}=\mbox{Span}\{\chi_iv_{\varepsilon_j^i}^{l,\pm}\phi_l:l=0,
\ldots,n,\;j=1,\ldots,k_i\mbox{ and }i=1,2\}.
$$
Here $v_{\varepsilon_j^i}^{l,\pm}$ and $\phi_l$ are defined in Section \ref{sec-jacobi-operator}.
\end{definition}

We remark here that the definition of the deficiency space above is slightly 
different from the one in \cite{cjs2024}. In the previous paper we consider the 
marked moduli space, but in the present setting we want to allow the puncture points to 
move, and so we must consider the unmarked moduli space. This difference forces us 
to include the next Fourier mode of Jacobi fields in the deficiency space. 

Given $w\in\mathcal W_{g_m}$ we can write
\begin{equation}\label{eq005}
    w=\sum_{i=1}^2\sum_{j=1}^{k_i}\sum_{l=0}^n\alpha_{i,j}^{l,\pm}v_{\varepsilon_j^i}^{l,\pm}\phi_l\chi_i,
\end{equation}

    where $\alpha_{i,j}^{l,\pm}$ are real numbers. For $j=1,\ldots,k_i$, at the ball 
    $B_{r_0}(q_j^i)$ we can write $g_m=(v_{\varepsilon_j^i}(t-T_j^i)+v)^\frac{4}{n-4}g_{\rm cyl}$, 
    for some decaying function $v$. 

\begin{proposition}[Linear Decomposition -- \cite{cjs2024}]\label{propo001} 
Let $1< \delta < \min\{ \gamma_{\varepsilon_j^i,n+1}:0 \leq j \leq k_i\}$,  
$u\in C^{4,\alpha}_\delta(M_i)$ and $\phi \in C^{0,\alpha}_{-\delta}(M_i)$ satisfying 
$L_{g_i}(u) = \phi$. Then there exist $w \in \mathcal{W}_{g_i}$ and $v\in 
C^{4,\alpha}_{-\delta}(M_i)$ such that $u=w+v$. 
\end{proposition}

For $1 < \delta < \min\{ \gamma_{\varepsilon_j^i,n+1}:1 \leq j \leq k_i, i=1,2\} $ the operator 
$$L_{g_i}:\mathcal W_{g_i}\oplus C^{4,\alpha}_{-\delta}(M_i)\to C_{-\delta}^{0,\alpha}(M_i)$$
is surjective. Using elliptic regularity and the fact that $L_{g_i}$ is formally self-adjoint, 
the injectivity of $L_{g_i}:C^{4,\alpha}_{-\delta}(M_i)\to C^{0,\alpha}_{-\delta}(M_i)$ 
implies $L_{g_i}:C^{4,\alpha}_{\delta}(M_i)\to C^{0,\alpha}_{\delta}(M_i)$ is surjective. 
Thus, Proposition \ref{propo001} implies that for all $f\in C^{0,\alpha}_{-\delta}(M_i)\subset 
C^{0,\alpha}_{\delta}(M_i)$ there exists $u\in \mathcal W_{g_i}\oplus C^{2,\alpha}_{-\delta}(M_i)$ 
such that $L_{g_i}(u)=f$. The kernel $B_{g_i}$ of the map $L_{g_i}:\mathcal W_{g_i}\oplus 
C^{4,\alpha}_{-\delta}(M_i)\to C_{-\delta}^{0,\alpha}(M_i)$ is called the {\it bounded null space}. 
The dimension of $B_{g_i}$ is $(k_i+1)(n+1)$, that is, each end contributes $n+1$ to the dimension.

\begin{remark}\label{remark001}
Since $B_{g_i}\in\mathcal W_{g_i}\oplus C^{4,\alpha}_{-\delta}(M_i)$, we can consider the 
natural projection $\mathcal{P}:B_{g_i}\to\mathcal W_{g_i}$. If $u,v\in B_{g_i}$ have the same 
image $\mathcal{P}(u)=\mathcal{P}(v)$, then $L_{g_i}(u-v)=0$ and $u-v\in 
C^{4,\alpha}_{-\delta}(M_i)$. Under the assumption that $g_i$ is unmarked nondegenrate, we 
obtain that $u=v$. So the map $\mathcal{P}:B_{g_i}\to\mathcal W_{g_i}$ is injective. We will 
consider $\tilde B_{g_i}=\mathcal{P}(B_{g_i})\subset\mathcal W_{g_i}$. Note that 
$\mathcal{P}:B_{g_i}\to\tilde B_{g_i}$ is a bijective map between finite dimensional spaces.    
\end{remark}

\begin{lemma}\label{lem002}
    Suppose there exists a one-parameter family of conformal metrics $g_{1-\eta}$ on $M_1$ with 
    constant $Q$-curvature $n(n^2-4)/8$ and asymptotic necksize of $\varepsilon+\eta$ on the 
    end at $q_0^1$. Given $w\in\mathcal W_{g_1}$, there exists $\Phi_1\in B_{g_1}$ such that 
    near the point $q_0^1$ we have
    $$|\Phi_1(t,\theta)+w(t,\theta)|\leq ce^{-\delta t},$$
    where $1 < \delta < \min\{ \gamma_{\varepsilon_j^i,n+1}:1 \leq j \leq k_i, i=1,2\} $.
\end{lemma}

\begin{proof} Let $g_{1-\eta}=F_{\eta}^{\frac{4}{n-4}}\mathring g$. Let $\Pi:\Ss^n
\backslash\{-q_0^1\}\to\R^n$ be the standard stereographic projection with pole at $-q_0^1$. 
We can suppose that $-q_0^1$ is a regular point of $g_{1-\eta}$ for all $\eta$. Then 
$(\Pi^{-1})^*g_{1-\eta}=f_{\eta}^{\frac{4}{n-4}}g_{\rm euc},$ where $f_\eta=
(F_\eta\circ \Pi^{-1})u_{sph}$ as in Remark \ref{gauge_remark}. By \eqref{eq018}, near the 
origin, which corresponds to $q_0^1$, as we did in \eqref{eq001}, we have
$$\mathfrak{F} (f_\eta  ) (t, \theta) = v_{\varepsilon_*+\eta} (t-T_0^1) + 
\mathcal{O}(e^{-\beta_\eta t}),$$
with $\beta_0=\beta_0^1>1$. Let \(w\in\mathcal{W}_{g_1}\) be an element of the deficiency space. 
Then \(w\) can be expressed as
\[w=\chi_1\sum_{j=0}^{k_1}\sum_{l=0}^n\alpha_{1,j}^{l,\pm}v_{\varepsilon_*}^{l,\pm}\phi_l,\]
where $\chi_1$ is as in Definition \ref{definition-defi-space}. Consider the following four 
families of functions. First, by assumption, we have the family of functions $F_{\eta}$ varying 
the asymptotic necksize of the Deluanay solution. Next, we have families obtained by applying 
the deformation $x\mapsto R^{\frac{n-4}{2}}f_\eta(Rx)$ and the Kelvin's transformation to $f_\eta$, 
as we did to obtain \eqref{deformed_del}. Geometrically, these two families correspond to 
translation of the necks along a Delaunay end and translation of the point at infinity, respectively. 
The final family of functions corresponds to translation at the origin. We denote these four 
families as follows, 
\begin{align*}
    &s\mapsto \mathfrak{F}(f_{\alpha_{1,0}^{0,-}s})\\
    &s\mapsto \mathfrak{F}(R(s)^{\frac{n-4}{2}}f_0(\cdot R(s))))\\
    &s\mapsto\mathfrak{F}(\mathbb{K} (\mathbb{K} (f_0) (\cdot - a(s)) )\\
    &s\mapsto\mathfrak{F}(f_0(\cdot+s\alpha_{1,0}^{l,-}e_l))
\end{align*}
where $R(0)=1, a(0)=0,$ and $\{e_l: 1\leq l\leq n\}$ forms a basis for $\R^{n}$. Let $R'(0)=\alpha_{1,0}^{0,+}$ and $a'(0)=(\alpha_{1,0}^{1,+},\ldots,\alpha_{1,0}^{n,+})$ and take the derivative of the families with respect to $s$ at $s=0$. Near the point $q_0^1$
\begin{align*}
    \frac{d}{ds}\mathfrak{F}(f_{\alpha_{1,0}^{0,-}s})&=\alpha_{1,0}^{0,-}v_{\varepsilon_*}^{0,-}+\mathcal{O}(e^{-\beta_0 t})\\
    \frac{d}{ds}\mathfrak{F}(R(s)^{\frac{n-4}{2}}f_0(\cdot R(s))))&=\alpha_{1,0}^{0,+}v_{\varepsilon_*}^{0,+}+\mathcal{O}(e^{-\beta_0 t})\\
    \frac{d}{ds}\mathfrak{F}(\mathbb{K} (\mathbb{K} (f_0) (\cdot - a(s)) )&=\sum\limits_{l=1}^{n}\alpha_{1,0}^{l,+}v_{\varepsilon_*}^{l,+}+\mathcal{O}(e^{-\beta_0 t})\\
    \frac{d}{ds}\mathfrak{F}(f_0(\cdot+s\alpha_{1,0}^{l,-}e_l))&=\alpha_{1,0}^{l,-}v_{\varepsilon_*}^{l,-}+\mathcal{O}(e^{-\beta_0 t})
\end{align*}

This implies that, near $q_0^1$, each derivative is asymptotic to $\alpha_{1,j}^{l,\pm}v_{\varepsilon_*}^{l,\pm}\phi_l$. See Section \ref{sec-jacobi-operator}. Combining all of these transformations, we obtain a single family $\Psi_s$ whose derivative at $s=0$ we define to be $\Phi_1$. Then $\Phi_1$ is asymptotic to $w$ near the point $q_0^1$.
\end{proof}

We are now in a position to prove the existence of a bounded right inverse for the linearized 
operator $L_{g_m}:\mathcal W_{g_m}\oplus C^{4,\alpha}_{-\delta}(M)\to C_{-\delta}^{0,\alpha}(M)$, 
for some $\delta>1$. Recall that
\begin{align*}
L_{g_m}(u) & =P_{g_m}(u)+\frac{n(n+4)(n^2-4)}{16}u\\
    & =\Delta_{g_m}^2u + \operatorname{div}_{g_m} \left (\frac{4}{n-2} 
    \operatorname{Ric}_{g_m} (\nabla u, \cdot) -
        \frac{(n-2)^2 + 4}{2(n-1)(n-2)} R_{g_m} du \right )\\
        & -\frac{n(n^2-4)}{2} u+\frac{n-4}{2}\psi u.
\end{align*}

\begin{proposition}\label{propo002}
    Suppose both $g_i$ are unmarked nondegenerate, and suppose there exists a one-parameter 
    family of conformal metrics $g_{1-\eta}$ on $M_1$ with constant $Q$-curvature $n(n^2-4)/8$ 
    and asymptotic necksize of $\varepsilon+\eta$ on the end at $q_0^1$. Then for $1 < \delta 
    < \min\{ \gamma_{\varepsilon_j^i,n+1}:1 \leq j \leq k_i, i=1,2\} $ there exists an $m_0>0$ 
    such that for all $m\geq m_0$ there exists an operator 
    \[G_{m}:C^{0,\alpha}_{-\delta}(M)\to \mathcal W_{g_m}\oplus C^{4,\alpha}_{-\delta}(M),\]
uniformly bounded in $m$, such that $L_{g_m}\circ G_m=Id$.
\end{proposition}

\begin{proof}
First, the fact that $L_{g_i}:\mathcal W_{g_i}\oplus C^{4,\alpha}_{-\delta}(M_i)\to 
C_{-\delta}^{0,\alpha}(M_i)$ is surjective, implies that this map has a bounded right inverse. 

Let $f\in C^{0,\alpha}_{-\delta}(M)$. Consider the cutoff function $\chi$ as defined 
in \eqref{eq003}, {\it i.e.} $\chi\equiv 1$ in $M_1\backslash 
B_{r_0e^{-T_0^1-(m+1/4)T_{\varepsilon_*}}}(q_0^1)$ and $\chi\equiv 0$ in 
$B_{r_0e^{-T_0^1-(m+3/4)T_{\varepsilon_*}}}(q_0^1)$.  
See Figure \ref{fig-cutoff-functions} for a sketch of the cutoff function $\chi$. Consider $\chi f$ 
as a function in $M_1$, extended as zero. Thus, there exists $w_1+v_1\in \mathcal W_{g_1}\oplus 
C^{4,\alpha}_{-\delta}(M_1)$ such that 
\begin{equation}\label{eq016}
    L_{g_1}(w_1+v_1)=\chi f
\end{equation}
and
\begin{equation}\label{eq012}
    \|w_1+v_1\|_{\mathcal W_{g_1}\oplus C^{4,\alpha}_{-\delta}(M_1)}\leq c\|\chi f\|_{C^{0,\alpha}_{-\delta}(M_1)}.
\end{equation}

 Since $w_1\in\mathcal W_{g_1}$, by the definition of the deficiency space 
 (Definition \ref{definition-defi-space}), in $B_{r_0}(q_0^1)$ we have that $w_1(t,\theta)$ is 
 asymptotic to $\sum\limits_{l=0}^n \alpha_{1,0}^{l,\pm}v_{\varepsilon_*}^{l,\pm}(t)
 \phi_l(\theta)$. By Lemma \ref{lem002} we can find $\Phi_1\in B_{g_1}$ such that 
 $\Phi_1+w_1= O(e^{-\delta t})$. By \eqref{eq012}, the fact that $B_{g_1}$ is canonically 
 identified with $\tilde B_{g_1}$ (Remark \ref{remark001}) and that both are finite 
 dimensional, we obtain
$$|w_1+v_1+\Phi_1|(t,\theta)\leq c\|\chi f\|_{C^{0,\alpha}_{-\delta}(M_1)} e^{-\delta t}$$
in $B_{r_0}(q_0^1)$, and
\begin{equation}\label{eq014}
     \|\Phi_1\|_{\mathcal W_{g_1}\oplus C^{4,\alpha}_{-\delta}(M_1)}\leq 
     c\|\chi f\|_{C^{0,\alpha}_{-\delta}(M_1)}.
\end{equation}

Similarly there exists $w_2+v_2\in\mathcal W_{g_2}\oplus C^{4,\alpha}_{-\delta}(M_2)$ such that
\begin{equation}\label{eq017}
    L_{g_2}(w_2+v_2)=(1-\chi) f
\end{equation}
and
\begin{equation}\label{eq013}
    \|w_2+v_2\|_{\mathcal W_{g_2}\oplus C^{4,\alpha}_{-\delta}(M_2)}\leq 
    c\|(1-\chi )f\|_{C^{0,\alpha}_{-\delta}(M_2)}.
\end{equation}

Here we will use Lemma \ref{lem002} again. By definition, at a small ball centered at the 
point $q_0^2$, we have
$$w_2(\tau,\theta)=\sum\limits_{l=0}^n \alpha_{2,0}^{l,\pm}v_\varepsilon^{l,\pm}(\tau)\phi_l(\theta).$$ 
Since the metrics $g_1$ and $g_2$ are asymptotic to the same Delaunay metric on 
$\hat{\mathbb{A}}_m$, there exists $\Phi_2\in B_{g_1}$ such that 
\begin{equation}\label{eq004}
\Phi_2(t,\theta)+\sum\alpha_{2,0}^{l,\pm}v_\varepsilon^{l,\pm}(t)\phi_l(\theta)=O(e^{-\delta t}).    
\end{equation}
By \eqref{eq013} and again the fact that $B_{g_1}$ is canonically identified with 
$\tilde B_{g_1}$ we obtain
\begin{equation}\label{eq015}
     \|\Phi_2\|_{\mathcal W_{g_1}\oplus C^{4,\alpha}_{-\delta}(M_1)}\leq 
     c\|(1-\chi) f\|_{C^{0,\alpha}_{-\delta}(M_2)}.
\end{equation}

Now consider two cutoff functions $\eta_1$ and $\eta_2$ defined in $M$ (See 
Figure \ref{fig-cutoff-functions}) such that
$$\eta_1(p)=\begin{cases}
    1, & p\in M_1\backslash \B_{r_0e^{-(T_0^1+(m+3/4)T_{\varepsilon_*})}}(q_0^1)\\
    0, & p\in M_2\backslash \B_{r_0e^{-(T_0^2+mT_{\varepsilon_*})}}(q_0^2)
\end{cases}$$
and
$$\eta_2(p)=\begin{cases}
    1, & p\in M_1\backslash \B_{r_0e^{-(T_0^2+(m+3/4)T_{\varepsilon_*})}}(q_0^2)\\
    0, & p\in M_1\backslash \B_{r_0e^{-(T_0^1+mT_{\varepsilon_*})}}(q_0^1)
\end{cases}$$
Define the operator $\hat G_m:C^{0,\alpha}_{-\delta}(M)\to \mathcal W_{g_m}\oplus 
C^{4,\alpha}_{-\delta}(M)$  by 
$$\hat{G}_m(f)=\eta_1(w_1+v_1+\Phi_1)+\eta_2\Phi_2+\eta_2(w_2+v_2).$$

\begin{figure}[!ht]
\begin{center}
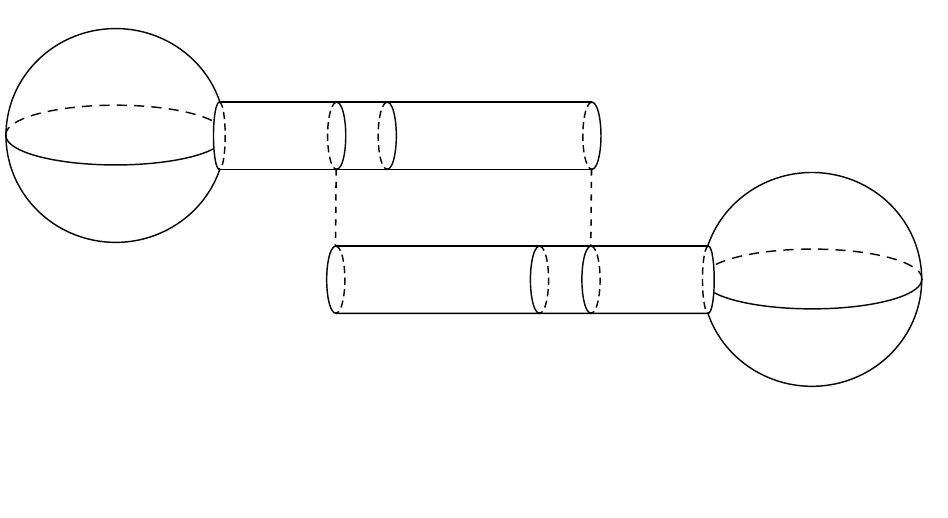
\caption{The punctured ball $B_{r_0}(q_0^i)\backslash\{q_0^i\}$ in $\Ss^n$ is identified with 
the half cylinder $(0,+\infty)\times\Ss^{n-1}$.}\label{fig-cutoff-functions}
\end{center}
\end{figure}

The result will follow after we prove the following two estimates
\begin{enumerate}
    \item[(i)] $\|\hat G_m(f)\|_{\mathcal W_{g_m}\oplus C^{4,\alpha}_{-\delta}(M)}\leq 
    c\|f\|_{C^{0,\alpha}_{-\delta}(M)}$,
    \item[(ii)] $\|L_{g_m}\circ\hat  G_m(f)-f\|_{C^{0,\alpha}_{-\delta}(M)}\leq 
    \frac{1}{2}\|f\|_{C^{0,\alpha}_{-\delta}(M)}$,
\end{enumerate}
for $m$ sufficiently large. While (i) shows us that $\hat G_m$ is bounded, the inequality 
(ii) implies that $L_{g_m}\circ \hat G_m$ has a bounded right inverse given by
$(L_{g_m}\circ \hat G_m)^{-1}:=\sum_{i=0}^\infty(I- L_{g_m}\circ\hat G_m)^i,$
with norm $\|(L_{g_m}\circ \hat G_m)^{-1}\|\leq 1$. Therefore, a uniformly bounded right 
inverse of $L_{g_m}$ is $G_m:=\hat G_m\circ (L_{g_m}\circ \hat G_m)^{-1}$.

The inequality (i) follows by \eqref{eq012}, \eqref{eq014}, \eqref{eq013} and \eqref{eq015}. 
To show (ii),  observe that in $M_i\backslash \B_{r_0 e^{-(T_0^i+mT_{\varepsilon_*})}}(q_0^i)$
 we have $\eta_i\equiv 1$ and the other cutoff function $\eta$ is identically zero. Moreover, within 
 this region, the cutoff function $\chi$ is identically $0$ or $1$. This directly implies that
 $L_{g_m}\circ\hat G_m(f)=f$ in these regions. In the annulus 
 $\B_{r_0e^{-(T_0^1+(m+1/4)T_{\varepsilon_*})}}(q_0^1)\backslash 
 \B_{r_0e^{-(T_0^1+(m+3/4)T_{\varepsilon_*})}}(q_0^1)\sim 
 \B_{r_0e^{-(T_0^2+(m+1/4)T_{\varepsilon_*})}}(q_0^2)\backslash 
 \B_{r_0e^{-(T_0^2+(m+3/4)T_{\varepsilon_*})}}(q_0^2)$, where $\eta_1$ and $\eta_2$ are 
 both equal to 1, we have
$$\hat{G}_m(f)(t,\theta)=(w_1+v_1+\Phi_1+\Phi_2)(t,\theta)+(w_2+v_2)(\tau,\theta),$$
where $t$ and $\tau$ are related by \eqref{approx_soln_identification}. 
By \eqref{eq001}, \eqref{annulus_metric3}, \eqref{eq016} and \eqref{eq017} we obtain
$$\| L_{g_m}\circ\hat G_m(f)-f\|_{C^{4,\alpha}_\delta(M)}=
\mathcal{O}(e^{-m\gamma})\|f\|_{C^{4,\alpha}_\delta(M)},$$
for some $\gamma>0$. It remains to estimate in the regions which corresponds 
to $t\in[T_0^1+mT_{\varepsilon_*}, T_0^1+(m+1/4)T_{\varepsilon_*}]\cup
[T_0^1+(m+3/4)T_{\varepsilon_*},T_0^1+(m+1)T_{\varepsilon_*}]$, where $\nabla\eta_1$, 
$\nabla\eta_2$ and $\nabla\chi$ are nonzero. In the region 
$t\in[T_0^1+mT_{\varepsilon_*}, T_0^1+(m+1/4)T_{\varepsilon_*}]$, we have
$$\hat{G}_m(f)(t,\theta)=(w_1+v_1+\Phi_1)(t,\theta)+\eta_2(\tau)(w_2+v_2+\Phi_2(t,\cdot))(\tau,\theta),$$
where $t$ and $\tau$ are related by \eqref{approx_soln_identification}. Using \eqref{eq004}, 
this implies that
\begin{align*}
    |L_{g_m}\circ \hat G_m(f)-f|(t,\theta) & =|L_{g_m}(\eta_2(w_2+v_2+
    \Phi_2(t,\cdot))|(\tau,\theta) \leq Ce^{-m\gamma} \|f\|_{C^{4,\alpha}_\delta(M)}.
\end{align*}
One can similarly estimate $|L_{g_m}\circ \hat G_m(f)-f|$ in the region 
$t\in[T_0^1+(m+3/4)T_{\varepsilon_*}, T_0^1+(m+1)T_{\varepsilon_*}]$. This finishes the proof.
\end{proof}

\section{Nonlinear Analysis} \label{sec:nonlinear} 
\subsection{The geometric deformations}

The approximate solution $g_m$ is conformal to the spherical metric outside of $\mathbb A_m$. So, 
for each $i\in\{1,2\}$ and $j\in\{1,\ldots,k_i\}$, in $\B_{r_0}(q_j^i)\backslash\{q_j^i\}$, we can 
write $g_m=U_{ij}^{\frac{4}{n-4}}\mathring g=u_{ij}^{\frac{4}{n-4}}g_{\rm euc}=g_i$. By 
expansion \eqref{asymp_emden_fowler}, there exists parameters $\varepsilon_j^i>0$, $T_j^i\in\R$ 
and $a_j^i\in\R^n$ such that near the point $q_j^i$, in normal coordinates centered at $q_j^i$, it holds
$$u_{ij}=u_{\varepsilon_j^i,\exp(T_j^i), a_j^i} + v_j^i,$$
with $|v_j^i(x)|=\mathcal{O}(|x|^{\frac{4-n}{4}+\beta_j^i})$, for some $\beta_j^i>1$. Note here that 
we are using lower case for normal coordinates since it is analogous to $\R^n$.
 
For $w\in \mathcal{W}_{g_m}$ with $\alpha_{i,j}^{0,-}$ small, $w$ satisfying \eqref{eq005}, we 
define a function $\widetilde u_{ij}$ in $\B_{r_0}(q_j^i)\backslash\{q_j^i\}$ as
$$\widetilde u_{ij}(x)=\begin{cases}
    u_{ij}(x), & \mbox{ if } |x|\geq r_0/2\\
    (1-\eta(x))u_{\varepsilon_j^i,\exp(T_j^i), a_j^i}(x)+\\
    +\eta(x)u_{\varepsilon_j^i+\alpha_{i,j}^{0,-},\exp(T_j^i+\alpha_{i,j}^{0,+}), 
    a_j^i+\alpha_{i,j}^+}(x+\alpha_{i,j}^-)+ v_j^i(x), &\mbox{ if } r_0/4\leq |x|\leq r_0/2\\
  u_{\varepsilon_j^i+\alpha_{i,j}^{0,-},\exp(T_j^i+\alpha_{i,j}^{0,+}), 
  a_j^i+\alpha_{i,j}^+}(x+\alpha_{i,j}^-) + v_j^i(x),  &\mbox{ if } 0<|x|\leq r_0/4,
\end{cases}$$
where  $\alpha_{i,j}^\pm=(\alpha_{i,j}^{l,\pm})\in\R^{n}$, and $\eta$ is a 
smooth cutoff function identically zero on $\B_{\frac{r_0}{4}}(0)$ and 
identically equal to one outside of $B_{\frac{r_0}{2}}(0)$ in normal 
coordinates centered at $q_j^i$. Thus, we define a metric $g_m(w)$ in $M$ 
such that $g_m(w)=g_m$ outside the balls $B_{r_0}(q_j^i)$, and $g_m(w)
=\widetilde u_{ij}^{\frac{4}{n-4}}g_{\rm euc}$ on the punctured balls 
$B_{r_0}(q_j^i)\backslash\{q_j^i\}$. Geometrically this denotes first 
deforming the neck size from $\varepsilon_j^i$ to $\varepsilon_j^i 
+\alpha_{i,j}^{0,-}$ and then translating along the Delaunay end by 
$\alpha_{i,j}^{0,+}+T_j^i$. This is followed by transferring to $\R^n$ 
via stereographic projection, translating the origin by $\alpha_{i,j}^-$, 
and translating infinity by $\alpha_{i,j}^+$. Note that when $w=0$ we have 
exactly $g_m(0)=g_m.$ Note that the singular points will be moved in this 
procedure if any of the $\alpha_{i,j}^{l,\pm}$ is nonzero for any $l\not=0$. 
We sketch the effect of these geometric deformation in 
Figure \ref{fig-geometric-deformations}.

\begin{figure}[!ht]
\begin{center}
\begingroup%
  \makeatletter%
  \providecommand\color[2][]{%
    \errmessage{(Inkscape) Color is used for the text in Inkscape, but the package 'color.sty' is not loaded}%
    \renewcommand\color[2][]{}%
  }%
  \providecommand\transparent[1]{%
    \errmessage{(Inkscape) Transparency is used (non-zero) for the text in Inkscape, but the package 'transparent.sty' is not loaded}%
    \renewcommand\transparent[1]{}%
  }%
  \providecommand\rotatebox[2]{#2}%
  \newcommand*\fsize{\dimexpr\f@size pt\relax}%
  \newcommand*\lineheight[1]{\fontsize{\fsize}{#1\fsize}\selectfont}%
  \ifx\svgwidth\undefined%
    \setlength{\unitlength}{261.78544893bp}%
    \ifx\svgscale\undefined%
      \relax%
    \else%
      \setlength{\unitlength}{\unitlength * \real{\svgscale}}%
    \fi%
  \else%
    \setlength{\unitlength}{\svgwidth}%
  \fi%
  \global\let\svgwidth\undefined%
  \global\let\svgscale\undefined%
  \makeatother%
  \begin{picture}(1,0.75729648)%
    \lineheight{1}%
    \setlength\tabcolsep{0pt}%
    \put(0,0){\includegraphics[width=\unitlength,page=1]{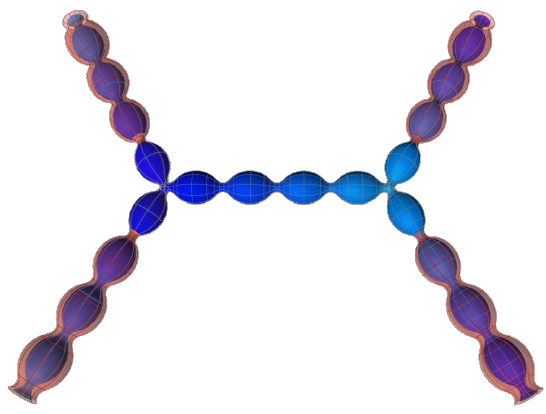}}%
    \put(0.49464094,0.17432046){\color[rgb]{0,0,0}\makebox(0,0)[t]{\lineheight{1.25}\smash{\begin{tabular}[t]{c}deformed Delaunay\\asymptotes\end{tabular}}}}%
    \put(0.50485891,0.64524297){\color[rgb]{0,0,0}\makebox(0,0)[t]{\lineheight{1.25}\smash{\begin{tabular}[t]{c}transition regions\\where $Q_g\not=\frac{n(n^2-4)}{8}$\end{tabular}}}}%
    \put(0,0){\includegraphics[width=\unitlength,page=2]{desenhando.pdf}}%
  \end{picture}%
\endgroup%

\caption{Deforming a metric by an element of the deficiency space.}\label{fig-geometric-deformations}
\end{center}
\end{figure}

The nonlinear equation we now wish to solve is 
\begin{equation} \label{goal_eqn} 
0 = \mathcal{N} (w,v) = \mathcal{N}_{g_m(w)} (1+v) .\end{equation}  
Observe that, since in general the puncture points move, the two metrics 
$g_m = g_m(0)$ and $g_m (w)$ will not be conformal. However, by our construction we 
have 
$$\left. \frac{d}{dt} \right|_{t=0} \mathcal{N}_{g_m(tw)}(1+tv)
= \mathcal{L}_{g_m} (w) + \mathcal{L}_{g_m} (v) = \mathcal{L}_{g_m}(w+v).$$
Thus we can write \eqref{goal_eqn} as 
\begin{eqnarray} \label{eq006}
0 & = & \mathcal{N}_{g_m(w)} (1+v) = (\mathcal{N}_{g_m(w)}(1+v) - \mathcal{N}_{g_m}(1+v))
+ \mathcal{N}_{g_m}(1+v) \\ \nonumber 
& = & \mathcal{L}_{g_m}(w) + \mathcal{R}_1 + \mathcal{N}_{g_m}(1+v) = 
\mathcal{L}_{g_m}(w+v) + \psi + \mathcal{R}_1 + \mathcal{R}_2,
\end{eqnarray} 
where $\psi$ is given by Lemma \ref{lem001} and $\mathcal{R}_1$ are 
$\mathcal{R}_2$ are quadratically small remainder terms. 

\subsection{Solving the gluing problem with a contraction}\label{sec-fixed-point}
We wish to find \((w,v)\in {\mathcal{W}_{g_m}\oplus C^{4,\alpha}_{-\delta}(M)}\), such 
that the $Q$-curvature of \(g=(1+v)^{\frac{4}{n-4}}g_m(w)\) is equal to the 
constant \(n(n^2-4)/8\). To do 
that we will consider the map $\mathcal N:\mathcal B_r\to C^{0,\alpha}_{-\delta}(M)$ given 
by \eqref{eq006}, where $\mathcal B_r\subset \mathcal W_{g_m}\oplus C^{4,\alpha}_{-\delta}(M)$ 
is a small ball centered at $(0,0)$ with radius $r$.

Using Taylor expansion we can decompose
\begin{equation}\label{eq21}
    \mathcal N(w,v)=\mathcal N(0,0)+\mathcal{L}(w,v)+\mathcal R(w,v),
\end{equation}
where \(\mathcal{L}(w,v)\) is the linearization of \(\mathcal N(w,v)\) and \(\mathcal R(w,v)\) is 
the remainder. By \eqref{eq006} we have $\mathcal L(w,v)=L_{g_m}(w+v)+\frac{n-4}{2}\psi w$. Since 
$\psi$ has compact support in $\mathbb{A}_m$, it follows by Lemma \ref{lem001}, 
Proposition \ref{propo002} and a perturbation argument that $\mathcal L:\mathcal W_{g_m}
\oplus C^{4,\alpha}_{-\delta}(M)\to C^{0,\alpha}_{-\delta}(M)$ has a uniformly bounded right 
inverse for all $m$ large enough, which we will denote as $\mathcal G$. This implies that the 
problem is equivalent to finding a fixed point of the operator $\mathcal K_m:\mathcal 
B_r\to \mathcal B_r$ given by
\begin{equation}\label{eq008}
    \mathcal K_m(w,v)=-\mathcal G(\mathcal N(0,0)+\mathcal{R}(w,v)).
\end{equation}

We will show that $\mathcal K_m$ is well defined, i.e., for all $(w,v)\in \mathcal B_r$, the 
right hand side of \eqref{eq008} belongs to $\mathcal B_r$, and also that $\mathcal K_m$ 
is a contraction. To do that it is enough to show that 
\begin{equation}\label{eq009}
    \|\mathcal K_m(0,0)\|\leq\frac{1}{2}r
\quad\mbox{ and }\quad \|\mathcal K_m(w_1, v_1)-\mathcal K_m(w_0, v_0)\|\leq 
\frac{1}{2}\|(w_1,v_1)-(w_0,v_0)\|.
\end{equation}
 In fact, if this is the case, then for all $(w,v)\in\mathcal B_r$ we have
$$\|\mathcal K_m(w,v)\|\leq \|\mathcal K_m(0,0)\|+\|\mathcal K_m(0,0)-\mathcal K_m(w,v)\|\leq r.$$

To show the first inequality in \eqref{eq009}, recall that $\mathcal G$ is bounded 
independently of $m$. Note that by \eqref{paneitz-branson}, \eqref{eq020} and  \eqref{eq006} 
we have $\mathcal N(0,0)=\frac{n-4}{2}\psi$. Thus, by Lemma \ref{lem001} we get
$$\|\mathcal K_m(0,0)\|\leq C\|\mathcal N(0,0)\|\leq C\|\psi\|\leq Ce^{-m\beta T_{\varepsilon_*}}.$$

Also, we have
$$\|\mathcal K_m(w_1, v_1)-\mathcal K_m(w_0, v_0)\|\leq 
C\|\mathcal R(w_1, v_1)-\mathcal R(w_0, v_0)\|$$

Now, using \eqref{eq21} and the Fundamental Theorem of Calculus, we express the remainder
\begin{eqnarray*}
    \mathcal R(w_1, v_1)-\mathcal R(w_0, v_0)
    -\mathcal{L}(u_1-u_0) & = & \int_0^1\left(\frac{d}{dt}\mathcal N(u_t)-\mathcal{L}(u_1-u_0)\right)dt\\
   &=& \int_0^1\int_0^1 \frac{d}{ds}\mathcal L^{su_t}((w_1,v_1)-(w_0, v_0)) ds dt.
\end{eqnarray*}
Here, $u_t=(w_1,v_1)+(1-t)((w_0, v_0)-(w_1, v_1))$ and $\mathcal L^{u_t}(u_1-u_0)
=\frac{d}{dt}\mathcal N(u_t)$. This implies that 
$$\|\mathcal R(u_1)-\mathcal R(u_0)\|\leq C\|u_1-u_0\|^2\leq C\max\{\|u_0\|,\|u_1\|\}\|u_1-u_0\|,$$
giving us the following proposition as an immediate consequence. 

\begin{proposition}\label{propo003}
   There exist $m_0>0$ and $r_0>0$ such that for all integer $m>m_0$ the map     
   $\mathcal K_m:\mathcal B_r\to \mathcal B_r$ given by \eqref{eq008}
    is a contraction. Consequently, \(\mathcal K_m\) has a fixed point $u\in\mathcal B_r$ 
    with norm bounded by $Ce^{-m\beta T_{\varepsilon_*}}$.
\end{proposition}

Let \(u\in\mathcal B_r\subset\mathcal  W_{g_m}\oplus{C^{4,\alpha}_{-\delta}(M)}\) be the 
unique fixed point of \(\mathcal K_m\) given by the previous proposition. Then see 
that \[u=\mathcal K_m(u)=-\mathcal G(\mathcal N(0)+\mathcal R(u)).\] Applying the 
linearized operator \(\mathcal L\) to both sides, we get \[\mathcal L(u)=
\mathcal L(-\mathcal G(\mathcal N(0)+\mathcal R(u)))=-\mathcal N(0)-\mathcal R(u),\] and we 
see that \(u=(w,v)\) is our solution to the nonlinear gluing problem, i.e. the 
metric \(\widetilde g_m=(1+v)^{\frac{4}{n-4}}g_m(w)\) is complete and admits constant 
$Q$-curvature on all of \(M\).

\section{Nondegeneracy} \label{sec:nondegen} 

To provide a full description of the type of solution obtained through the previous argument, the 
idea is to follow through the ideas developed in \cite{MR1712628,MPU2,Jesse} to show that, 
for sufficiently large $m$, it is possible to conclude that the metric obtained through 
the gluing argument $g_{m}$ is unmarked nondegenerate. 

\subsection{Auxiliary Lemmas}

Recall the sympletic form defined in \cite{cjs2024}. Let 
$g_1 = U_1^{\frac{4}{n-4}} \overset{\circ}{g}$  defined on $\Ss^{n-1}\backslash\Lambda_1$, as 
in Section \ref{approx_sltn}, and transfer $g_1$ to $\R^n \backslash \Gamma$ using 
stereographic projection, rewriting $g_1 = u_1^{\frac{4}{n-4}} g_{\rm euc}$ with $u_1
= (U_1\circ \Pi^{-1}) u_{\rm sph}$. For any sufficiently small $r>0$ we define $\Omega_r = 
\R^n \backslash \left ( \bigcup_{j=0}^{k_1} \B_r(q_j^1) \right )$ and 
\begin{equation} \label{defn_symp_form}
\omega(v,w) = \lim_{r \searrow 0} \int_{\Omega_r} (v L_{u_1} (w) - 
w L_{u_1} (v) )d\mu_0.\end{equation} 
Here $d\mu_0$ is the Euclidean volume element and $L_{u_1}$ is the Jacobi operator of $g_1$, which 
is defined as in \eqref{eq011}.

\begin{lemma} \label{decay}
    Suppose $v\in B_{g_{1}}$ decays like $e^{-\delta_{j} t}$ near all singular points $q_{j}^1$ for 
    some $\delta_{j}> 1$, except $q_{0}^1$. Then $v$ is asymptotic to 
    $\alpha_{1,0}^{0,+}v_{\varepsilon_*}^{0,+}$ near $q_{0}^1$, for some $\alpha_{1,0}^{0,+}\in\R$. 
    In addition, if there exists $w\in B_{g_{1}}$ asymptotic to ${v}_{\varepsilon_*}^{0,-}$ 
    near $q_{0}^1$, then v decays at least like $e^{-\delta t}$ near $q_{0}^1$, for some $\delta >1$.
\end{lemma}

\begin{proof}
We know that $v$ can be expressed as \eqref{eq005}. Suppose that $\alpha_{1,0}^{0,-}\not=0$. By 
Lemma \ref{lem002} there exists $w\in B_{g_1}$ such that $w$ is asymptotic to 
$v_{\varepsilon_*}^{0,+}$ near $q_0^1$ and smooth in $\Ss^n\backslash\{q_0^1\}$. On one hand, 
since $v,w \in B_{g_{1}}$, then $L_{g_{1}}(v) = L_{g_{1}}(w) = 0$, which implies that 
$\omega(v,w)= 0$. On the other hand, using \eqref{eq011}, \eqref{defn_symp_form} and 
integration by parts, we have
 \begin{equation*} \label{symp_form2} 
 \omega(v,w) = \lim_{r \searrow 0} \sum_{j=0}^k 
 \int_{\partial \B_r(q_j^1)} (w \partial_r \Delta_{g_1} v - v \partial_r \Delta_{g_1} w
 + \partial_r v \Delta_{g_1} w - \partial_r w \Delta_{g_1} v )d\mu_0.\end{equation*}
Since $v$ decays like $e^{-\delta_jt}$ near the singular points $q_j^1$, for all $j\not=0$, and 
$w$ is smooth in $\Ss^n\backslash\{q_0^1\}$, we show that 
$$\omega(v,w)=\alpha_{1,0}^{0,-}\frac{d}{d\varepsilon}\mathcal H_\varepsilon\not=0,$$
where $\mathcal H_\varepsilon$ is defined in \eqref{eq002}, leading to a contradiction. See 
the proof of Theorem 16 in \cite{cjs2024} for details.
All the other cases, $\alpha_{1,0}^{l,\pm}=0$ for $l\in \{1,\dots, n\}$, follow through a 
similar argument. 
\end{proof}

\begin{lemma} Let $\delta \in\left(1, \gamma_{n+1}(\varepsilon)\right)$. Then there exists an 
operator
$$
H_{T, \varepsilon}: C_{-\delta}^{0, \alpha}\left([-T, T] \times \Ss^{n-1}\right) 
\rightarrow C_{-\delta}^{4, \alpha}\left([-T, T] \times \Ss^{n-1}\right),
$$
uniformly bounded in $T$, such that $u=H_{T, \varepsilon}(f)$ is a solution of 
$L_{g_{\varepsilon}}(u)=f$.
\end{lemma}
\begin{proof}
The idea of the proof is to use a pertubation argument, similar to what was used on 
Proposition \ref{propo002}. Consider $\beta$ to be a cut-function on $[-T,T] \times \Ss^{n-1}$ 
such that $\beta(t,\theta)=1$ for $t\leq -1$, and $\beta(t,\theta)=0$ for $t\geq 1$.

Since the metric $g_{\varepsilon}$ is unmarked nondegenerate, let $u_{1} \in 
C^{4,\alpha}([-T, \infty)\times \Ss^{n-1})$ be a solution of $L_{g_1}(u_{1}) = \beta f$. 
Similarly, $u_{2} \in C^{4,\alpha}([-\infty, T]\times \Ss^{n-1})$ let $L_{g_{\varepsilon}}(u_{2}) 
= (1-\beta)f$. Define $\hat{H}_{T,\varepsilon}(f) = \beta u_{1} + (1-\beta) u_{2}$ as the 
approximation of the inverse. Using the perturbation argument, the proof is concluded.
\end{proof}

\subsection{Nondegeneracy of the solution}

The remaining details of the proof of Theorem \ref{main_thm} closely follow the ideas 
presented in \cite{Jesse}. We include them here for the sake of completeness.

\begin{proposition}
    The metrics $\widetilde g_m$ constructed in Section \ref{sec-fixed-point} are unmarked 
    nondegenerate for all $m>0$ large enough.
\end{proposition}
\begin{proof} We begin with some general comments about the mapping properties of the 
Jacobi operator $L_g$. It follows from definitions that if $L_{g}:W^{k+4,2}_{-\delta} \rightarrow 
W^{k,2}_{-\delta}$ is injective for some $\delta >1$ then $L_g : W^{k+4,2}_{-\delta'} 
\rightarrow W^{k,2}_{-\delta'}$ is also injective for any $\delta'>\delta$. Furthermore, 
one can use the twisted operator of Section 3.6 of \cite{jesse2020} to show that 
if $L_g:W^{k+4,2}_{-\delta} \rightarrow W^{k,2}_{-\delta}$ is injective some some 
$\delta >1$ then it is injective for {\bf each} $\delta>1$. Roughly speaking, this follows from 
the properties of the meromorphic inverse of the twisted operator. One realizes changes 
of the weight in the function space by shifting a certain contour integral constructing 
this meromorphic inverse operator. So long as the contour in question does not surround 
a pole, the two inverse operators agree. 

We prove the nondegeneracy of the solutions constructed in Section \ref{sec:nonlinear} 
by contradiction. After the paragraph above, we can fix $\delta >1$. Assume there exist 
$m_k \rightarrow \infty$ such that $g_k = g_{m_k}$ is unmarked degenerate, which in turn 
there exist $u_k \in C^{4,\alpha}_{-\delta} (M, g_k)$ such that $L_{g_k} (u_k) = 0$. As in 
\eqref{eq022}, we decompose 
$$M= M_1^c \cup M_2^c \cup \left ( \bigcup_{j=1}^{k_1} B_{r_0} (q_j^1) \backslash \{ q_j^1\}
\right ) \cup \left ( \bigcup_{j=1}^{k_2} B_{r_0} (q_j^2) \backslash \{ q_j^2\}
\right ) \cup \widehat{\mathbb{A}}_{m_k} . $$
We will consider the normalization
\begin{equation*}
    \sup_{M}\rho_{k}^{-1}|u_{k}| = 1,
\end{equation*}
where $\rho_{k}$ is given by
$$\rho_{k}(p) = 
\begin{cases}
    1,   & p \in M_{i}^{c}\\
    |q-q_j^i|^{\frac{4-n}{2} +\delta'}, & p\in B_{r_0} (q_j^i) \backslash \{ q_j^i \} \\ 
    \mathfrak{F}^{-1} \left ( \frac{\cosh^{\delta}(m_{k}T_{\varepsilon_*})}{\cosh^{\delta}t}\right ), 
    & p=(t,\theta) \in \hat{\mathbb{A}}_{m_{k}},
\end{cases}$$
for some $1<\delta'<\delta$. Near the puncture points, we can write the weight function 
in Emden-Fowler coordinates as 
$$\mathfrak{F}(\rho) (t,\theta) = e^{-\delta' t} , \qquad \mathfrak{F}(\rho)(\tau,\theta) 
= e^{-\delta'\tau}.$$ 
We also use \eqref{eq023} to identify the extended annulus $\widehat{\mathbb{A}}_{m_k}$ 
with $(-T_0^1-(m_k+1/2)T_{\varepsilon_*},T_0^2+(m_k+1/2)T_{\varepsilon_*})\times\Ss^{n-1}$. 

Since $\rho_k^{-1}(p)|u_k(p)|\to 0$ when $p\to q_j^i$, for all $j\not=0$, we can choose a 
point $p_{k}\in M$ such that $\rho_k(p_k)^{-1}|u_k(p_k)|=1$. The proof follows by obtaining 
contradictions depending on the place in $M$ where the maximum point is located. The main idea 
is to construct a sequence of solutions converging to a Jacobi field that cannot exist in the 
weighted space we consider. Now we divide in cases.

\medskip
\noindent{\bf Case 1:} $p_{k}=\left(t_{k}, \theta_{k}\right) \in \hat{\mathbb{A}}_{m_{k}}$ 
and $\left|t_{k}\right|$ is bounded.
\medskip

By \eqref{annulus_metric3} and Proposition \ref{propo003} we have that $g_k$ converges 
to $g_{\varepsilon_*}=v_{\varepsilon}^{\frac{4}{n-4}}g_{\rm cyl}$. Define $\widetilde u_k$ 
in  $\hat{\mathbb{A}}_{m_{k}}$ as $\tilde{u}_{k}(t, \theta)=
\left(\cosh  (m_{k}T_{\varepsilon_*})\right)^{-\delta} u_{k}(t, \theta)$. Since $|t_k|$ is 
bounded and $|\widetilde u_k(t,\theta)|\leq(\cosh t)^{-\delta}$, we can extract a subsequence 
such that  $\left(t_{k}, \theta_{k}\right) \rightarrow(t_0, \theta_0)$ and $\widetilde u_k$ 
converges uniformly on compact sets of $\R\times\Ss^{n-1}$ to a smooth function $\overline u$. 
Moreover,  $|\overline u(t,\theta)|\leq (\cosh t)^{-\delta}$, with equality holding at 
$(t_0,\theta_0)$.  Therefore, $L_{g_{\varepsilon_*}}(\overline u)=0$ and $\overline u$ is 
nontrivial. By decomposing $\overline u$ into the eigenvalues of the Laplacian in the 
sphere $\Ss^{n-1}$, we obtain a contradiction to the fact that the indicial roots determine 
the rates of exponential growth of the solutions.

\medskip
\noindent{\bf Case 2:} $p_{k}=\left(t_{k}, \theta_{k}\right) \in \hat{\mathbb{A}}_{m_{k}}$ 
and $\left|t_{k}\right|$ is unbounded.
\medskip

 Assume that $t_{k}<0$ (the case $t_k>0$ is similar). In this case we restrict to the part 
 of $\hat{\mathbb{A}}_{m_{k}}$ identified with $\left[-T_0^1-(m_k+1/2)T_{\varepsilon_*}-t_{k},
 -t_k\right] \times \Ss^{n-1}$. Up to a subsequence, we can suppose that 
 $-T_0^1-(m_k+1/2)T_{\varepsilon_*}-t_{k}\to a_0\in[-\infty,0]$. Define
$$
\widetilde{u}_{k}(t, \theta)=\frac{\cosh^\delta t_k}{\cosh ^\delta 
(m_kT_{\varepsilon_*})}u_k\left(t+t_{k}, \theta\right).
$$
Thus $\left|\widetilde{u}_{k}\left(0, \theta_{k}\right)\right|=1$. Besides, for all $t\in\mathbb R$ 
we have $\cosh t_k\leq 2e^{t}\cosh(t+t_k)$. This implies that
$$
\left|\widetilde u_{k}(t, \theta)\right| \leq  2^\delta e^{\delta t}.
$$
Again, we can take a subsequence of $\widetilde u_k$ converging smoothly on compact sets of 
$(a_0,+\infty)\times \Ss^{n-1}$ to $\overline {u}$ satisfying $L_{g_{\varepsilon_*}}
(\overline u)=0$, $|\overline u(0,\theta_0)|=1$ and $|\overline u(t,\theta)|
\leq 2^\delta e^{\delta t}$. As before, we get a contradiction.

\medskip
\noindent{\bf Case 3:} $p_{k}=\left(t_{k}, \theta_{k}\right)\in B_{r_0}(q_j^i) \backslash\{q_j^i\}$, 
for some $j=1, \ldots, k_i$ and some $i=1,2$, and $t_{k} \rightarrow \infty$.
\medskip

This case cannot occur because $u_k/\rho_k \rightarrow 0$ as $p \rightarrow q_j^i$ for 
each $i\in \{ 1,2\}$ and $j \in \{ 1, \dots, k_i\}$. 

\medskip
\noindent{\bf Case 4:} $p_{k} \in \Omega_1$, where $\Omega_1$ is some fixed compact set 
containing $M_1^c$.  
\medskip

Since $\Omega_1$ is compact, we can assume $p_k$ converges to some $\overline{p}\in \Omega_1$. 
Notice $\rho_{k}$ is bounded and bounded away from 0 in $\Omega_1$. Define $\tilde{u}_k=
\rho_k^{-1}(p_k)u_k$ in the set $$
M_1^c \cup\left(\bigcup_1^{k_1} B_{r_0}\left(q_j^{1}\right) \backslash
\left\{q_j^{1}\right\}\right) \cup\left\{(t, \theta) \in 
\hat{\mathbb{A}}_{m_{k}} \mid t<0\right\}.
$$
Then $\tilde{u}_k$ is uniformly bounded and $\tilde{u}_k(p_k)=1$, so up to a subsequence, $u_k$ 
converges uniformly to a non-trivial function $\overline{u}$ in $M_1$ where 
$L_{g_1}(\overline{u})=0$. Also, $\overline{u}$ decays exponentially near all $q^1_j$ except $q^1_0$. 
By Lemma \ref{lem002}, there exists some $w\in B_{g_1}$ which is asymptotic to 
$v_{\varepsilon_*}^{0,-}$. Applying Lemma \ref{decay}, $\bar{u}$ decays like $e^{-\delta t}$ near 
$q^1_0$ for some $\delta>1$, which contradicts the fact that $\left(M_1, g_1\right)$ is 
unmarked nondegenerate.

\medskip

\noindent{\bf Case 5:} \( p_k \in \Omega_2 \), where \( \Omega_2 \) is a fixed compact set 
containing \( M_2^c \).

\medskip
 Again, taking the restriction to
\[
M_2^c \cup \left(\bigcup_{j=1}^{k_2} B_{r_0}(q_j^2) \setminus \{q_j^2\} \right) \cup 
\left\{(t, \theta) \in \hat{\mathbb{A}}_{m_{k}} \mid t > 0 \right\},
\]  
we can extract a convergent subsequence. However, in this case, the only conclusion we obtain is 
that \( \bar{u} \) is asymptotic to \( a v_{\varepsilon}^{0,+} \) near \( q_0^2 \), which is 
not yet sufficient to reach a contradiction.

To develop this case further, rescale the problem so that $a=1$ and fix $R_0>0$. Then, there 
is an $k_0$ depending on $R_0$ such that for $k \geq k_0$
$$\left\|u_{k}-v_{\varepsilon}^{0,+}\right\|_{C^{4, \alpha}\left(\left[m_k T_{\varepsilon_*}
-R_0-1, m_k T_{\varepsilon_*}\right] \times \Ss^{n-1}\right)}=O\left(e^{-\delta R_0}\right)
$$
Moreover, if we restrict $u_{k}$ to
$M_1^c \cup\left(\bigcup_1^{k_1} B_{r_0}(q_j^1) \backslash \{q_j^1\}\right) 
\cup\left\{(t, \theta) \in \hat{\mathbb{A}}_{m_{k}} \mid t<0\right\},
$ as we previously argued, $u_{k}$ converges uniformly to zero.
Thus for $k \geq k_0$,
$$\left\|u_k\right\|_{C^{4, \alpha}\left(\left[-m_k T_{\varepsilon_*},
-m_k T_{\varepsilon_*}+R_{0}+1\right] \times \Ss^{n-1}\right)}=O\left(e^{-\delta R_0}\right).$$

Recall that we can write the approximate metric $g_{k}$ in $\hat{\mathbb{A}}_{m_{k}}$ 
as $\left(u_{\varepsilon}+v_{k}\right)^{\frac{4}{n-4}}\left(d t^2+d \theta^2\right),
$ where $v_{k}$ has the following decay
$$
\left|v_{k}(t, \theta)\right|=O\left(\frac{\cosh ^{\gamma_{n+1}(\varepsilon)} t}
{\cosh ^{\gamma_{n+1}(\varepsilon)} m_k T_{\varepsilon_*}}\right).
$$
Consequently $L_{g_k}-L_{g_{\varepsilon}}$, when restricted to $\hat{\mathbb{A}}_{m_{k}}$, is a 
fourth order differential operator whose coefficients are $O\left(\cosh ^\mu t / 
\cosh ^\mu m_k T_{\varepsilon_* }\right)$ on $\hat{\mathbb{A}}_{m_{k}}$ for some 
$\mu \in\left(\delta, \gamma_{n+1}(\varepsilon)\right)$, which implies
$$
L_{g_{\varepsilon}}\left(u_{k}\right)(t, \theta)=O\left(\frac{\cosh ^{\mu-\delta} t}
{\cosh ^{\mu-\delta} m_k T_{\varepsilon_*}}\right) \quad \mbox{ in } \hat{\mathbb{A}}_{m_{k}}.
$$
In other words, $L_{g_\varepsilon}\left(u_{k}\right) \in 
C_{\delta-\mu}^{0, \alpha}\left(\left[-m_k T_{\varepsilon_*}+R_0, m_k T_{\varepsilon_*}
-R_0\right] \times \Ss^{n-1}\right)$ and
$$
\left\|L_{g_\varepsilon}\left(u_{k}\right)
\right\|_{C_{\delta-\mu}^{0, \alpha}
\left(\left[-m_k T_{\varepsilon_*}+R_0, m_k T_{\varepsilon_*}-R_0\right] \times 
\Ss^{n-1}\right)}= O\left(e^{(\delta-\mu) R_0}\right).
$$
Let $\tilde{u}_{k}=H_{m_k T_{\varepsilon_*}-R_{0}, 
\varepsilon}\left(L_{g_{\varepsilon}}\left(u_{k}\right)\right)$. Then
$$
\begin{aligned}
 0&=\int_{\left[-m_k T_{\varepsilon_*}+R_0, m_k T_{\varepsilon_*}-R_0\right] 
 \times \Ss^{n-1}} \tilde{u}_{k} L_{g_{\varepsilon}}\left(v_{\varepsilon}^{0,-}\right)
 -v_t^{0,-} L_{g_\varepsilon}\left(\tilde{u}_{k}\right) \\
 = & \int_{\left\{m_k T_{\varepsilon_*}-R_{0}\right\} \times \Ss^{n-1}}(\bar{u}_{k} \partial_r 
 \Delta_{g_{\varepsilon}} v_{\varepsilon}^{0,-} - v_{\varepsilon}^{0,-} 
 \partial_r \Delta_{g_{\varepsilon}} \bar{u}_{k}
 + \partial_r v_{\varepsilon}^{0,-} \Delta_{g_{\varepsilon}} \bar{u}_{k}  
 - \partial_r \bar{u}_{k} \Delta_{g_{\varepsilon}} v_{\varepsilon}^{0,-} ) \\ 
 &  \qquad -\int_{\left\{-m_k T_{\varepsilon_*}+R_0\right\} \times \Ss^{n-1}}
 (\bar{u}_{k} \partial_r \Delta_{g_{\varepsilon}} v_{\varepsilon}^{0,-} 
 - v_{\varepsilon}^{0,-} \partial_r \Delta_{g_{\varepsilon}} \bar{u}_{k}
 + \partial_r v_{\varepsilon}^{0,-} \Delta_{g_{\varepsilon}} \bar{u}_{k} - 
 \partial_r \bar{u}_{k} \Delta_{g_{\varepsilon}} v_{\varepsilon}^{0,-} ) \\
=& 1+ O\left(e^{-\delta R_0}+e^{(\delta-\mu) R_0}\right).
\end{aligned}
$$
which is a contradiction and concludes the proof of Theorem \ref{main_thm}.

\end{proof}
\section{Topology of the unmarked moduli space} \label{sec:top} 

In this section, we complete the proof of Theorem \ref{noncontractible_thm}, 
showing that the loop we construct cannot be contractible. We 
first use Liouville's theorem to describe the space of conformal classes on a 
finitely punctured sphere. 

\begin{lemma} \label{conf_classes} 
Let $n\geq 3$ and let $k \geq 4$. The space of conformal classes of 
complete Riemannian metrics on $\Ss^n \backslash \Lambda$ where $\# \Lambda 
= k$ is precisely 
$$\mathrm{Conf}_k = (\Ss^n \backslash \{ p_0, p_1, p_2\})^{k-3} 
\backslash \mathrm{Diag}, $$
where $\{ p_0, p_1, p_2\}$ is any (fixed) triple of points in 
$\Ss^n$ and $\mathrm{Diag}$ is the full diagonal, {\it i.e.} 
the set of all tuples such that some point $p\in \Ss^n$ appears twice. 
\end{lemma} 

\begin{proof} 
Let $\Lambda$ be a set of $k$ punctures associated to a Riemannian 
metric $g$. By Lioville's theorem, there 
exists a unique global conformal transformation of $\Ss^n$ mapping 
the first three points in $\Lambda$ to the fixed triple $\{ p_0, p_1, p_2\}$, 
so we can identify the conformal class of $g$ with a $(k-3)$-tuple of 
distinct points in $\Ss^n \backslash \{ p_0, p_1, p_2\}$. By the uniqueness 
statement in Liouville's theorem, if this process identifies two metrics with 
the same $(k-3)$-tuple of points they must be conformally equivalent. 
\end{proof}

The space $(\Ss^n \backslash \{ p_0, p_1, p_2\})^{k-3} 
\backslash \mathrm{Diag}$ is a smooth manifold of 
dimension $(k-3)n$. In the proof of Theorem \ref{nontriv_homotopy} 
below we describe exactly how it is not contactible by constructing a 
homotopy class that is not equivalent to a point. 

At this point we describe the family of metrics we create using 
end-to-end gluing more carefully. Let $M_1 = \Ss^n \backslash \{ p_0, 
\dots, p_{k-2}\}$ and let $g_1 = U_1^{\frac{4}{n-4}} \overset{\circ}{g}$ 
be a complete, conformally flat metric with constant $Q$-curvature that 
is unmarked nondegenerate. Next let $M_2 = \{ p_0, p_{k-1}, p_k\}$ and 
$g_2 = U_2^{\frac{4}{n-4}} \overset{\circ}{g}\in \mathcal{M}_3$. We assume 
the asymptotic necksizes of $g_1$ and $g_2$ at the common puncture $p_0$ 
agree, and that $g_2$ belongs to a one-parameter family of metrics in 
$\mathcal{M}_3$ that changes the asymptotic necksize at $p_0$ to first 
order. Then we glue end-to-end, obtaining a metric whose singular 
set is $\{ p_1, \dots, p_k\}$. However, within this construction we 
can rotate the summand $(M_2, g_2)$ by any element of $SO(n-1)$, 
{\it i.e.} by any rotation fixing the common axis of $g_1$ and $g_2$ 
at the gluing site, producing a family of metrics in $\mathcal{M}_k$. 
We denote the resulting family of metrics by $\mathfrak{m}\subset 
\mathcal{M}_k$. We also 
let $\overline{\mathfrak{m}} = \pi \circ \mathfrak{m}$, where 
$$\pi:\mathcal{M}_k \rightarrow \mathrm{Conf}_k$$
is the forgetful map sending a singular, constant $Q$-curvature 
metric to its conformal class. At this point it is convenient to compose with a 
family of conformal maps so that $p_k$, $p_1$ and $p_2$ all remain fixed. 
Abusing notation slightly, 
we still denote these families as $\mathfrak{m} \subset \mathcal{M}_k$ 
and $\overline{\mathfrak{m}} \subset \mathrm{Conf}_k$. Observe that, by Liouville's 
theorem, none of $p_j$ for $3 \leq j \leq k-1$ can be fixed in this family. 

\begin{theorem} \label{nontriv_homotopy} 
There does not exist a homotopy in $\mathcal{M}_k$ 
contracting the family $\mathfrak{m}$ constructed above to a point. 
\end{theorem} 

\begin{proof} 
Suppose there exists such a homotopy 
$$F:[0,1] \times SO(n-1) \rightarrow \mathcal{M}_k, \quad 
F(0,A) = \mathfrak{m}(A), \quad F(1,A) = \mathrm{const}$$
and let $\overline{F} = \pi \circ F$.  
Then $\overline{F}$ must be a homotopy contracting $\overline{\mathfrak{m}}$
to a point. However, 
the action of $SO(n-1)$ on the puncture points moves both of  $p_j$  
over a small sphere surrounding $p_k$ for each $j =3,\dots, k-1$. 
If one could contract this family to 
a point in $\mathrm{Conf}_k$, then at some time in the homotopy the points 
$p_k(t) = \overline{F}(p_k,t)$ and $p_{j}(t) = \overline{F}(p_{j},t)$
would coincide. However, the resulting configuration now 
lies outside the set of allowable conformal classes on a $k$-punctured sphere, 
and so is impossible.
\end{proof}

\bibliography{end-to-end}
\bibliographystyle{abbrv}

\end{document}